\documentclass{amsart}
\usepackage{amssymb}
\usepackage[nobysame]{amsrefs}
\usepackage{mathrsfs}
\usepackage[usesnames,svgnames]{xcolor}
\usepackage[sc]{mathpazo}
\usepackage{tikz,pgf}
	\pgfdeclarelayer{background}
	\pgfsetlayers{background,main}
\usepackage[colorlinks=true,
	urlcolor=MidnightBlue,
	citecolor=DarkGreen]{hyperref}
\usepackage{todonotes}
\usepackage{multirow}
\usepackage{enumerate}
\usepackage{subcaption}
\usepackage{scalerel,stackengine}
\makeatletter
\def\namedlabel#1#2{\begingroup
    #2%
    \def\@currentlabel{#2}%
    \phantomsection\label{#1}\endgroup
}
\makeatother
\newtheorem{theorem}{Theorem}[section]
\newtheorem{conjecture}[theorem]{Conjecture}
\newtheorem{proposition}[theorem]{Proposition}
\newtheorem{corollary}[theorem]{Corollary}
\newtheorem{lemma}[theorem]{Lemma}
\newtheorem{definition}[theorem]{Definition}

\theoremstyle{remark}
\newtheorem{example}[theorem]{Example}
\newtheorem{remark}[theorem]{Remark}
\newcommand{\defn}[1]{{\color{DarkGreen}\emph{#1}}}
\newcommand{\ie}{\text{i.e.}\;}
\newcommand{\st}{^{\text{st}}}

\newcommand{\Sym}{\mathfrak{S}}
\newcommand{\Alt}{\mathfrak{A}}
\newcommand{\Braid}{\mathfrak{B}}
\newcommand{\Alto}{\mathfrak{A}^{o}}
\newcommand{\nc}{N\!C}

\newcommand{\pnc}{\mathcal{N\!C}}
\newcommand{\enc}{O\!N\!C}
\newcommand{\penc}{\mathcal{O\!N\!C}}

\newcommand{\rk}{\operatorname{rk}}
\newcommand{\red}{\text{Red}}
\newcommand{\id}{e}

\newcommand{\Poset}{\mathcal{P}}
\newcommand{\Int}{\mathcal{I}}

\newcommand{\MaxChains}{\mathcal{M}}

\newcommand{\Rank}{\mathcal{R}}
\newcommand{\ZetaPol}{\mathcal{Z}}

\newcommand{\ww}{\mathbf{w}}
\newcommand{\xx}{\mathbf{x}}
\newcommand{\yy}{\mathbf{y}}
\newcommand{\zz}{\mathbf{z}}
\newcommand{\btw}{\mathbf{2}}
\newcommand{\bth}{\mathbf{3}}
\stackMath
\newcommand\reallywidehat[1]{%
\savestack{\tmpbox}{\stretchto{%
  \scaleto{%
    \scalerel*[\widthof{\ensuremath{#1}}]{\kern-.6pt\bigwedge\kern-.6pt}%
    {\rule[-\textheight/2]{1ex}{\textheight}}%WIDTH-LIMITED BIG WEDGE
  }{\textheight}% 
}{0.5ex}}%
\stackon[1pt]{#1}{\tmpbox}%
}
\newcommand{\oc}{\operatorname{ocyc}}
\newcommand{\cyc}{\operatorname{cyc}}
\newcommand{\la}{\ell_{\bth}}
\newcommand{\lt}{\ell_{\btw}}
\newcommand{\leqa}{\leq_{\bth}}
\newcommand{\leqt}{\leq_{\btw}}
\newcommand{\seq}{\mathbf{s}}
\newcommand{\Gen}{\mathcal{T}}
\newcommand{\GJbij}{\Phi}
\newcommand{\Cay}{\operatorname{Cay}}
\title{A Poset Structure on the Alternating Group Generated by $3$-Cycles}

\begin{document}

\author{Henri M\"uhle}
\address{Technische Universit{\"a}t Dresden, Institut f{\"u}r Algebra, Zellescher Weg 12--14, 01069 Dresden, Germany.}
\email{henri.muehle@tu-dresden.de}

\author{Philippe Nadeau}
\address{Univ Lyon, CNRS, Universit\'e Claude Bernard Lyon 1, UMR5208, Institut Camille Jordan, F-69622 Villeurbanne, France}
\email{nadeau@math.univ-lyon1.fr}

\keywords{Symmetric group, Alternating group, Noncrossing partitions, Hurwitz action, zeta polynomial}
\subjclass[2010]{06A07 (primary), and 05A10, 05E15, 20B35 (secondary)}

\begin{abstract}
	We investigate the poset structure on the alternating group that arises when the latter is generated by $3$-cycles. We study intervals in this poset and give several enumerative results, as well as a complete description of the orbits of the Hurwitz action on maximal chains. Our motivating example is the well-studied absolute order arising when the symmetric group is generated by transpositions, i.e. $2$-cycles, and we compare our results to this case along the way. In particular, noncrossing partitions arise naturally in both settings.
\end{abstract}

\maketitle

%%%%%%%%%%%%%%%%%%%%%%
\section{Introduction}
	\label{sec:introduction}
%%%%%%%%%%%%%%%%%%%%%%
	
Given a group $G$ generated by a finite set $\Gen$, the \defn{(right) Cayley graph} $\Cay(G,\Gen)$ is perhaps one of the most fundamental geometric objects attached to it. Recall that the vertex set of $\Cay(G,\Gen)$ is $G$, and that its edges are of the form $(g,gt)$ for $g\in G,t\in\Gen$.  It comes with a natural graph distance which can be written as $d_{\Cay}(g,g')=\ell_{\Gen}(g^{-1}g')$. Here the \defn{length} $\ell_{\Gen}(g)$ is the minimum $k$ such that there exists a factorization $g=t_1t_2\cdots t_k$ where each $t_i$ is in $\Gen$. Such factorizations are \defn{$\Gen$-reduced} and $\red_{\Gen}(g)$ denotes the set of all $\Gen$-reduced factorizations of $g$.

The relation defined by $u\leq_{\Gen}v$ if and only if $\ell_{\Gen}(v)=\ell_{\Gen}(u)+\ell_{\Gen}(u^{-1}v)$ is a partial order on $G$, graded by $\ell_\Gen$.  $\red_{\Gen}(g)$ is naturally identified with the set of maximal chains from $\id$ to $g$ in this poset. Geometrically, $u\leq_{\Gen}v$ holds if and only if $u$ occurs on a geodesic from the identity $\id$ to $v$ in $\Cay(G,\Gen)$.

If we require furthermore that $\Gen$ is closed under $G$-conjugation, then we can define a natural action of the braid group on $\ell_{\Gen}(g)$ strands on the set $\red_{\Gen}(g)$; the \defn{Hurwitz action}.  Informally, this action can be described as follows: the $i\text{th}$ generator of the braid group shifts the $(i+1)\st$ letter of an element in $\red_{\Gen}(g)$ one step to the left, and conjugates as it goes.  The Hurwitz action on closed intervals in $(G,\leq_{\Gen})$ was recently studied in \cite{muehle17connectivity}.  For specific groups, the Hurwitz action was studied for instance in \cites{baumeister17on,benitzhak03graph,bessis15finite,deligne74letter,hou08hurwitz,sia09hurwitz,wegener20on}.

\medskip

A well-studied example of this construction is the case where $G=\Sym_{N}$ is the symmetric group of all permutations of $[N]=\{1,2,\ldots,N\}$, and $\Gen$ is the set of all transpositions.  Whenever we refer to this case we replace the subscript ``$\Gen$'' by ``$\btw$'' in all of the above definitions.  It is a standard fact that $\lt(x)=N-\text{cyc}(x)$, where $\text{cyc}(x)$ denotes the number of cycles of $x$. The poset $(\Sym_{N},\leqt)$ was for instance studied in \cite{athanasiadis08absolute}.  Moreover, it was observed in \cite{biane97some} that the \defn{lattice of noncrossing partitions} arises as the principal order ideal $\pnc_{N}=[e,c]_{\btw}$ where $c$ is the $N$-cycle $c=(1\;2\;\ldots\;N)$.  See \cite{simion00noncrossing} for a survey on this lattice, and \cite{kreweras72sur} for some enumerative and structural properties.

It was in fact in this  setting that the Hurwitz action was first considered~\cite{hurwitz91ueber}.  One of the crucial properties, namely that for any $x\in\Sym_{N}$ the 
action is transitive on $\red_{\btw}(x)$, is motivated by the enumeration of branched covers of the Riemann sphere with $N$ given branch points.  

\medskip

In this article, we set out to study the poset $(G,\leq_{\Gen})$ in the case where $G=\Alt_{N}$ is the alternating group and $\Gen$ is the set of all $3$-cycles. In that case we replace the subscript ``$\Gen$'' by ``$\bth$''.  Note that several articles have been written about ordering the elements of $\Alt_{N}$ with respect to certain generating sets, mostly for the purpose of equipping $\Alt_{N}$ with a ``Coxeter-like'' structure; see~\cites{athanasiadis14absolute,brenti08alternating,mitsuhashi01the,regev04permutation,rotbart11generator}.  However, none of these orders fits into the framework described in the first paragraph of this section.  

We will show that the poset $(\Alt_{N},\la)$ has a rich combinatorial structure, which can be compared to the structure of $(\Sym_{N},\lt)$ mentioned before.  Our first main combinatorial result computes the \defn{zeta polynomial} of the closed interval $[\id,c]_{\bth}$ in $(\Alt_{N},\leqa)$, where $c=(1\;2\;\ldots\;N)$ is a long cycle for $N$ odd. The evaluation of the zeta polynomial at an integer $q$ yields the number of multichains of length $q-1$ in $[\id,c]_{\bth}$. 

\begin{theorem}\label{thm:main_zeta_polynomial}
	Let $N=2n+1$ be an odd positive integer, and let $c=(1\;2\;\ldots\;N)\in\Alt_{N}$.  The zeta polynomial of the interval $[\id,c]_{\bth}$ of $(\Alt_{N},\leqa)$ is given by
	\begin{displaymath}
		\ZetaPol(q) = \frac{q}{q(2n+1)-n}\binom{q(2n+1)-n}{n}.
	\end{displaymath}
\end{theorem}

We note that the zeta polynomial of $[\id,c]_{\btw}$ in $(\Sym_{N},\lt)$, where $c=(1\;2\;\ldots\;N)$ for any $N$ has a well-known formula \cite{kreweras72sur}*{Theorem 5} comparable to ours. 

We further show that the Hurwitz action is no longer transitive on $\red_{\bth}(x)$ for all $x\in\Alt_{N}$.  Our second main result computes the number of orbits depending on the number of even-length cycles of $x$.

\begin{theorem}\label{thm:main_hurwitz_orbits}
	Let $x\in\Alt_{N}$ for $N\geq 3$, and write $2k$ for its number of cycles of even length.  The Hurwitz action on $\red_{\bth}(x)$ has $(2k)_{k}=(k+1)(k+2)\cdots(2k)$ orbits.
\end{theorem}

We see in particular that the Hurwitz action is transitive on $\red_{\bth}(x)$ if and only if $x$ has only cycles of odd length.

\medskip

This article is organized as follows.  In Section~\ref{sec:preliminaries} we recall the necessary definitions and concepts.  In Section~\ref{sec:alternating} we recall the definition of the alternating group and prove some basic properties of the poset $(\Alt_{N},\leqa)$.  Subsequently, in Section~\ref{sec:alternating_noncrossing} we recall the definition of noncrossing partitions, and exhibit a subset of $\Alt_{N}$ whose combinatorial properties can be described with certain noncrossing partitions.  We prove Theorem~\ref{thm:main_zeta_polynomial} in Section~\ref{sec:alternating_combinatorics}, and establish further combinatorial results of certain intervals of ${(\Alt_{N},\leqa)}$.  In Section~\ref{sec:hurwitz} we prove Theorem~\ref{thm:main_hurwitz_orbits}, and we finish this article with Section~\ref{sec:extensions}, where we outline potential extensions of the work presented here.

%%%%%%%%%%%%%%%%%%%%%%
\section{Preliminaries}
	\label{sec:preliminaries}
%%%%%%%%%%%%%%%%%%%%%%

\subsection{Poset Terminology}
	\label{sec:posets}
Let $\Poset=(P,\leq)$ be a partially ordered set, or \defn{poset} for short. If $x < y$ in $\Poset$ and there exist no $z\in P$ such that $x<z<y$, then $y$ \defn{covers} $x$ and we write $x\lessdot y$.  The \defn{Hasse diagram} of $\Poset$ is the directed graph with vertex set $P$ formed by such \defn{cover relations}. For $x,y\in P$ with $x\leq y$, the set $\{p\in P\mid x\leq p\leq y\}$ is an \defn{interval}.  Let $\Int(\Poset)$ denote the set of all intervals of $\Poset$.
 
An \defn{($m$-)multichain} is a sequence $(x_{1},x_{2},\ldots,x_{m})$ of $m$ elements of $P$ such that $x_{1}\leq x_{2}\leq \cdots\leq x_{m}$.  If all elements in this sequence are distinct, we simply speak of an \defn{($m$-)chain}.  A chain $C$ (considered as a set) is \defn{maximal} if there is no $x\in P\setminus C$ such that $C\cup\{x\}$ is a chain.  Let $\MaxChains(\Poset)$ denote the set of maximal chains of $\Poset$.  

The poset $\Poset$ is \defn{graded} if its maximal chains all  have the same cardinality.  In that case we can assign a \defn{rank function} to $\Poset$, where the rank of an element $x$ is the maximum cardinality minus one of some chain from a minimal element (below $x$) to $x$.  Let $\Rank_{\Poset}(k)$ denote the set of elements of $\Poset$ having rank $k$.  The \defn{rank} of $\Poset$ is the rank of a maximal element of $\Poset$. 

Now let $\Poset$ be a \defn{bounded} poset, \ie $\Poset$ has a least element $\hat{0}$ and greatest element $\hat{1}$.  Moreover, let $\Poset$ be graded with rank function $\rk$.  Given an $m$-multichain $C=(x_{1},x_{2},\ldots,x_{m})$ of $\Poset$, the \defn{rank jump vector} of $C$ is $r(C)=(r_{1},r_{2},\ldots,r_{m+1})$, where $r_{i}=\text{rk}(x_{i})-\text{rk}(x_{i-1})$ for $i\in[m+1]$, and $x_{0}=\hat{0}, x_{m+1}=\hat{1}$.  

The \defn{M{\"o}bius function} of a poset $\Poset$ is the function $\mu\colon P\times P\to\mathbb{Z}$ recursively defined by
\begin{displaymath}
	\mu(x,y) = \begin{cases}1, & \text{if}\;x=y,\\ -\sum_{x\leq z<y}{\mu(x,z)}, & \text{if}\;x<y,\\ 0, & \text{otherwise}.\end{cases}
\end{displaymath}
If $\Poset$ is bounded, then we call the value $\mu(\hat{0},\hat{1})$ the \defn{M{\"o}bius number} of $\Poset$, and it is denoted by $\mu(\Poset)$.  

If $\Poset$ is finite, let $\ZetaPol_{\Poset}(m)$ denote the number of $(m-1)$-multichains of $\Poset$.  This is in fact a polynomial function of $m$ called the \defn{zeta polynomial} of $\Poset$, which encodes a lot of information about $\Poset$.  In particular, when $\Poset$ is bounded and graded of rank $n$, then $\mu(\Poset)=\ZetaPol_{\Poset}(-1)$, $\ZetaPol_{\Poset}(m)$ has degree $n$, and its leading coefficient is $\bigl\lvert\MaxChains(\Poset)\bigr\rvert/n!$; see~\cite{stanley11enumerative_vol1}*{Proposition~3.12.1}.

\subsection{Groups With Generating Sets Closed Under Conjugation}
	\label{sec:generated_groups}
Let $G$ be a group with identity $\id$, and let $\Gen\subseteq G$ be a generating set.  Any element $g\in G$ can be written as a word over the alphabet $\Gen$.  A shortest such word is \defn{$\Gen$-reduced}, and the length of a $\Gen$-reduced word for $g$ will be denoted by $\ell_{\Gen}(g)$.  This gives rise to a graded partial order on $G$ defined by 
\begin{align}\label{eq:order}
	x\leq_{\Gen}y\quad\text{if and only if}\quad\ell_{\Gen}(x)+\ell_{\Gen}(x^{-1}y)=\ell_{\Gen}(y);
\end{align}
the \defn{$\Gen$-prefix order}.  This definition has first appeared explicitly in \cite{brady01partial} in the case where $G$ is the symmetric group and $\Gen$ is the set of all transpositions.  It is also implicitly contained in \cite{biane97some} and has its origins perhaps in \cite{garside69braid}.  In other words, for $x,y\in G$ we have $x\leq_{\Gen}y$ if and only if $x$ lies on a geodesic from the identity to $y$ in the (right) Cayley graph of $G$ with respect to $\Gen$.  We write $[x,y]_{\Gen}$ for the interval between $x$ and $y$ in $(G,\leq_{\Gen})$.

The poset $(G,\leq_{\Gen})$ has rank function $\ell_{\Gen}$, and the following result states that when $\Gen$ is well behaved, this poset is in fact locally self dual.%  Moreover, we write $[\id,x]_{\Gen}$ for the principal order ideal of $(G,\leq_{\Gen})$ induced by $x\in G$.

\begin{proposition}[\cite{huang17absolute}*{Proposition~2.5}]\label{prop:generalized_kreweras}
	Let $G$ be a group, and suppose $\Gen\subseteq G$ is a generating set closed under $G$-conjugation.  Let $y,z\in G$ with $y\leq_{\Gen}z$. The permutation of $G$ defined by $x\mapsto yx^{-1}z$ restricts to a poset antiisomorphism $[y,z]_{\Gen}\to[y,z]_{\Gen}$. 
\end{proposition}

An important example of the maps from Proposition~\ref{prop:generalized_kreweras} is obtained in the special case $y=\id$; the resulting map $K_{z}(x)=x^{-1}z$ from $[\id,z]_{\Gen}$ to itself will be called the \defn{Kreweras complement}. 

\begin{corollary}\label{cor:shifting_intervals}
	Under the same hypotheses as in Proposition~\ref{prop:generalized_kreweras},  the map $K_{z}\circ K_{y^{-1}z}$ restricts to an isomorphism $[y,z]_{\Gen}\cong [\id,y^{-1}z]_{\Gen}$.
\end{corollary}

The next proposition is usually stated in the case $H=G$.

\begin{proposition}\label{prop:invariant_conjugation}
	Let $G$ be a group, and $H\leq G$ be a subgroup generated by some set $\Gen\subseteq H$.  If $\Gen$ is closed under $G$-conjugation, then $\ell_{\Gen}$ is invariant under $G$-conjugation, and if $x,x'\in H$ are $G$-conjugate, then the intervals $[\id,x]_{\Gen}$ and $[\id,x']_{\Gen}$ are isomorphic.
\end{proposition}
\begin{proof}
	Let $x\in H$ have $\ell_{\Gen}(x)=k$.  We can therefore write $x=t_{1}t_{2}\cdots t_{k}$ with $t_{i}\in \Gen$ for $i\in[k]$.  Now, let $x'=w^{-1}xw$ for some $w\in G$.  We obtain 
	\begin{displaymath}
		x'=(w^{-1}t_{1}w)(w^{-1}t_{2}w)\cdots(w^{-1}t_{k}w),
	\end{displaymath}
	and by assumption $w^{-1}t_{i}w\in \Gen$ for $i\in[k]$.  The assumption that $\ell_{\Gen}(x)=k$ now implies $\ell_{\Gen}(x')=k$, which proves the first claim.  The second claim follows then by definition of $\leq_{\Gen}$.
\end{proof}

\subsection{The Hurwitz Action}
	\label{sec:hurwitz_action}
 Let $\Braid_{k}$ denote the \defn{braid group on $k$ strands}, which admits the presentation
\begin{equation}
	\begin{split}
		\Braid_{k}=\bigl\langle\sigma_{1},\sigma_{2},\ldots,\sigma_{k-1}\mid & \;\sigma_{i}\sigma_{j}=\sigma_{j}\sigma_{i},\;\text{for}\;1\leq i,j<k,\;\text{and}\;\lvert i-j\rvert>1\;\text{and}\\
		& \;\sigma_{i}\sigma_{i+1}\sigma_{i}=\sigma_{i+1}\sigma_{i}\sigma_{i+1}\;\text{for}\;1\leq i<k-1\bigr\rangle.
	\end{split}
\end{equation}
When $\Gen$ is closed under $G$-conjugation, then $\Braid_{k}$ acts on the set of words over $\Gen$ of length $k$ as follows.  Let $\mathbf{x}=t_{1}t_{2}\cdots t_{k}$ be such a word.  The \defn{Hurwitz operators} $\sigma_{i}$ and $\sigma_{i}^{-1}$ act on $\mathbf{x}$ by
\begin{align}
	\sigma_{i}\cdot\mathbf{x} & = t_{1}t_{2}\cdots t_{i-1}t_{i+1}(t_{i+1}^{-1}t_{i}t_{i+1})t_{i+2}\cdots t_{k}\label{eq:hurwitz_left},\\
	\sigma_{i}^{-1}\cdot\mathbf{x} & = t_{1}t_{2}\cdots t_{i-1}(t_{i}t_{i+1}t_{i}^{-1})t_{i}t_{i+2}\cdots t_{k}\label{eq:hurwitz_right}.
\end{align}

\begin{remark}
	Note that if $t_i$ and $t_{i+1}$ commute then $\sigma_i$ and $\sigma_{i}^{-1}$ both perform this commutation in $\mathbf{x}$.
\end{remark}

It is straightforward to verify that the Hurwitz operators satisfy the defining relations for $\Braid_{k}$, so they indeed extend to a group action of $\Braid_{k}$, the \defn{Hurwitz action}. Let us say that two words are \defn{Hurwitz equivalent} if they are in the same orbit for this action. It is obvious from Equations~\eqref{eq:hurwitz_left} and \eqref{eq:hurwitz_right} that two equivalent words represent the same element of $G$. This implies that given $x\in G$ of length $k$, the Hurwitz action can be in particular studied on \defn{$\red_\Gen(x)$}, the set of all $\Gen$-reduced decompositions of $x$.

The Hurwitz action is closely related to the study and enumeration of ramified coverings of the $2$-sphere and probably appeared first in \cite{hurwitz91ueber}.  Such a ramified covering is essentially determined by a set of branching points, which give rise to a sequence of permutations, or a constellation, associated with this covering.  Theorem~1.2.28 in \cite{lando04graphs} states that two ramified coverings of the $2$-sphere are ``flexibly'' equivalent if and only if their associated constellations belong to the same Hurwitz orbit.  

The enumeration of Hurwitz orbits in subgroups of the symmetric group therefore yields information on the number of ``flexibly'' equivalent ramified coverings.  The number of Hurwitz orbits also plays a role for the enumeration of all ramified coverings (up to isomorphism).  If one such covering is given, then it is relatively easy to compute all coverings in its Hurwitz orbit.  If we want to compute all coverings, we thus need at least one representative from each Hurwitz orbit, which may be a non-trivial task at times. 

We explicitly describe and enumerate the Hurwitz orbits on $\red_\bth(x)$ for any element $x$ of the alternating group in Section~\ref{sec:hurwitz}.
	
%%%%%%%%%%%%%%%%%%%%%%
\section{The Alternating Group and $3$-Cycles}
	\label{sec:alternating}
%%%%%%%%%%%%%%%%%%%%%%
Let us now define the objects of our study.  Recall that the \defn{symmetric group} $\Sym_{N}$ is the group of all permutations of $[N]=\{1,2,\ldots,N\}$.  The \defn{support} of  $x\in\Sym_{N}$ is $\text{supp}(x)=\bigl\{i\in[N]\mid x(i)\neq i\bigr\}$.  A cycle is a permutation which moves the elements of its support cyclically. We say that a cycle is \defn{odd} (respectively \defn{even}) if its support has odd (respectively even) cardinality. Any permutation $x$ has a unique decomposition as a product of cycles with disjoint supports, up to the order of the factors. The set of \defn{cycles} of $x$ consists of the cycles of the previous decomposition together with the fixed points of $x$; where these fixed points are considered as odd cycles.

 $\Sym_{N}$ is naturally generated by the set of all transpositions, and the minimal number of transpositions needed to factor a permutation $x\in\Sym_{N}$ is given by the \defn{reflection length} 
\begin{equation}\label{eq:reflection_length}
	\lt(x) = N-\cyc(x), 
\end{equation}
where $\cyc(x)$ denotes the number of cycles of $x$. 

This well-known result is based on the following observation.  Let $(i\; j)$ be a transposition.  If $w\in\Sym_{n}$ has two distinct cycles containing $i$ and $j$, we may write $w=w'(\seq_i)(\seq_j)$ where $\seq_i$ and $\seq_j$ are sequences ending with $i$ and $j$ respectively.  Then $w\cdot (i\;j)=w'(\seq_i\; \seq_j)$, and so the number of cycles decreases by one. In this case we say that we \defn{join} the two cycles; more generally, given $m$ disjoint cycles of $w$, joining them in a new cycle consists in picking a starting element in each of them, concatenate the $m$ sequences thus obtained in any order, and consider the resulting sequence as a cycle. The inverse operation is called \defn{splitting} a cycle.

The \defn{alternating group} $\Alt_{N}$ is the subgroup of $\Sym_{N}$ that consists of all permutations with even reflection length. It is easily seen that $\Alt_{N}$ is generated by the set 

\begin{displaymath}
	C_{3,N}=\bigl\{(i\;j\;k)\mid 1\leq i,j,k\leq N,\lvert\{i,j,k\}\rvert=3\bigr\}
\end{displaymath}
of all $3$-cycles.  Since $C_{3,N}$ is closed under taking inverses, we can define a $C_{3,N}$-prefix order on $\Alt_{3}$ as in \eqref{eq:order} via the length function $\ell_{C_{3,N}}$.  Let us write $\la$ for this length function, and denote the resulting partial order by $\leqa$.  

Let $\oc(x)$ denote the number of odd cycles of $x\in\Alt_{N}$; by definition this number takes into account the fixed points of $x$.  The next proposition gives an explicit formula for $\la$, which nicely parallels that of $\lt$. To the best of our knowledge this result was first proved in~\cite{herzog76representation}. We give here a proof for completeness, and because it will give us Proposition~\ref{prop:summary_length} as a byproduct.

\begin{proposition}\label{prop:alternating_length}
	For any $x\in\Alt_{N}$, we have
	\begin{displaymath}
		\la(x)=\frac{N-\oc(x)}{2}.
	\end{displaymath}
\end{proposition}
\begin{proof}
	For $x\in\Alt_{N}$ define the quantity $\la^*(x)=\tfrac{N-\oc(x)}{2}$.  Let $a=(i\;j\;k)\in C_{3,N}$, and let $\zeta_{i},\zeta_{j}$, and $\zeta_{k}$ be the cycles of $x$ that contain $i,j$, and $k$, respectively. We let $x'$ be the product of all other cycles of $x$. 
	
	We first study the possible values for the difference $\la^*(xa)-\la^*(x)$. We need to distinguish three cases depending on the cardinality of $\{\zeta_{i},\zeta_{j},\zeta_{k}\}$. In the following we will denote by $\seq_t$ a sequence of integers ending with $t$.
	
	(i) Suppose first that the cycles $\zeta_{i},\zeta_{j},\zeta_{k}$ are pairwise distinct. Write $\zeta_i=(\seq_i)$, $\zeta_j=(\seq_j)$ and $\zeta_i=(\seq_k)$. Then 
	\begin{align*}
		x\cdot a & = x' (\seq_i) (\seq_j) (\seq_k) \cdot(i\; j\; k) = x' (\seq_i  \;\seq_j \;\seq_k).
	\end{align*}

	It follows that $\la^*(xa)-\la^*(x)=0$ if $\{\zeta_{i},\zeta_{j},\zeta_{k}\}$ contains zero or one odd cycle, and $\la^*(xa)-\la^*(x)=1$ if $\{\zeta_{i},\zeta_{j},\zeta_{k}\}$ contains two or three odd cycles.
	
	\medskip

	(ii) Suppose now that $\{\zeta_i,\zeta_j,\zeta_k\}$ has cardinality $2$. Because of the cyclic symmetry $(i\;j\;k)=(j\;k\;i)=(k\;i\;j)$, we can assume without loss of generality $\zeta_{i}=\zeta_{j}=:\zeta_{ij}$ and $\zeta_{ij}\neq\zeta_{k}$.  Write $\zeta_{ij}=(\seq_i \; \seq_j)$ and $\zeta_k=(\seq_k)$. Then
	\begin{align*}
		x\cdot a & = x' (\seq_i \; \seq_j) (\seq_k) \cdot(i\; j\; k) = x' (\seq_i) (\seq_j \;\seq_k).
	\end{align*}

	It follows that if $\lvert\seq_j\rvert$ is even, then $\la^*(xa)-\la^*(x)=0$. Let thus $\lvert\seq_j\rvert$ be odd.  Then $\la^*(xa)-\la^*(x)=0$ if $\lvert\seq_i\rvert$ and $\lvert\seq_k\rvert$ have the same parity, and $\la^*(xa)-\la^*(x)=-1$ if $\lvert\seq_i\rvert$ is odd and $\lvert\seq_k\rvert$ is even.  Moreover, $\la^*(xa)-\la^*(x)=1$ if $\lvert\seq_i\rvert$ is even and $\lvert\seq_k\rvert$ is odd.
	
	\medskip
	
	(iii) Assume finally $\zeta_{i}=\zeta_{j}=\zeta_{k}=:\zeta$. There are two possibilities for this cycle depending on the cyclic order of $i,j,k$: it can be either written as $(\seq_i \;\seq_j \;\seq_k)$ or $(\seq_i\;\seq_k\;\seq_j)$.  In the first case, 
	\begin{align*}
		x\cdot a & = x' (\seq_i \; \seq_j \;\seq_k) \cdot(i\; j\; k) = x' (\seq_i \;\seq_k\;\seq_j ).
	\end{align*}
	Clearly then $\la^*(xa)-\la^*(x)=0$. In the second case, 
	\begin{align*}
		x\cdot a & =x'  (\seq_i \;\seq_k\;\seq_j ) \cdot(i\; j\; k) = x' (\seq_i) (\seq_j) (\seq_k).
	\end{align*}
	It follows that $\la^*(xa)-\la^*(x)=0$ if the triple $\bigl(\lvert\seq_i\rvert,\lvert\seq_j\rvert,\lvert\seq_k\rvert\bigr)$ contains at most one odd integer, and $\la^*(xa)-\la^*(x)=-1$ otherwise.

	\medskip

	We now prove that $\la^*(x)=\la(x)$ for all $x\in\Alt_{N}$. First of all we notice that $\la^*(x)=0$ if and only if $x=\id$, since only the identity permutation possesses $N$ odd cycles.  By the case analysis performed above, we have $\la^*(xa)\leq \la^*(x)+1$ for any $x,a$. Therefore if $x$ is written as a product $x=a_1a_2\cdots a_k$, then  $\la^*(x)\leq k$ by immediate induction. By choosing $k$ minimal, we have $\la^*(x)\leq \la(x)$ for any $x\in\Alt_{N}$. 

	Now we claim that for any $x\neq\id$ we can find $a$ such that $\la^*(xa)=\la^*(x)-1$. Indeed,  either $x$ has an odd cycle $(u_1\;u_2\;u_3\;\ldots)$ of length at least $3$, and we choose $a=(u_3\;u_2\;u_1)$; or $x$ has two even cycles $(u_1\;u_2\;\ldots)$ and $(v_1\;v_2\;\ldots)$ and we choose $a=(u_2\;u_1\;v_1)$. By immediate induction we obtain the existence of a decomposition $x=a_1a_2\cdots a_k$ with $k=\la^*(x)$. This shows the reverse inequality $\la^*(x)\geq \la(x)$, and we conclude that $\la(x)=\la^*(x)$ for all $x\in\Alt_{N}$.  
\end{proof}

If $t$ and $t'$ are distinct integers that belong to the same cycle of $x$, define $r(t,t')$ to be the smallest $r\geq 0$ such that $x^r(t)=t'$. Equivalently, $r(t,t')$ is one more than the number of elements appearing between $t$ and $t'$ in the cycle notation of $x$.  

\begin{proposition}\label{prop:summary_length}
	Let $x\in \Alt_N$, $a=(i\;j\;k)\in C_{3,N}$ and $\zeta_t$ be the cycle of $x$ containing $t$ for $t\in\{i,j,k\}$. 
	\begin{itemize}
		\item $\la(xa)=\la(x)+1$ if and only if one of the following holds:
		\begin{itemize}
			\item $\zeta_i,\zeta_j,\zeta_k$ are pairwise disjoint and at least two of them are odd;
			\item $\{\zeta_i,\zeta_j,\zeta_k\}$ has cardinality $2$, say $\zeta_i=\zeta_j\neq \zeta_k$, both cycles are odd, and $r(i,j)$ is odd.
		\end{itemize}  
		\item $\la(xa)=\la(x)-1$ if and only if one of the following holds:
		\begin{itemize}
			\item $\zeta_i=\zeta_j=\zeta_k=\zeta$; $i,k,j$ appear in this cyclic order in $\zeta$, and 
			at least two elements of the list $r(i,k),r(k,j),r(j,i)$ are odd;
			\item  $\{\zeta_i,\zeta_j,\zeta_k\}$ has cardinality $2$, say $\zeta_i=\zeta_j\neq \zeta_k$, both cycles are even, and $r(i,j)$ is odd. 
		\end{itemize} 
		\item $\la(xa)=\la(x)$ in all other cases.
	\end{itemize}
\end{proposition}

The statement follows immediately from the case analysis performed in the proof of Proposition~\ref{prop:alternating_length}. As a corollary we get a description of the cover relations of $(\Alt_{N},\leqa)$, which fall into three categories.

\begin{corollary}\label{cor:covers}
	Let $y\in\Alt_{N}$. Then an element $x\in\Alt_{N}$ satisfies $x\lessdot_\bth y$ if and only if it is obtained by one of the following operations on $y$.
	\begin{enumerate}%[(i)]
		\item Pick an odd cycle of $y$ and split it into three odd cycles. \label{it:oddo}
		\item Pick an even cycle of $y$ and split it into two odd cycles and one even cycle. \label{it:odde}
		\item Pick two even cycles of $y$, split one into two odd cycles, and join one of these odd cycles with the second even cycle.\label{it:even}
	\end{enumerate}
\end{corollary}

\begin{example} Let us give examples of all three types of covers. These were written on purpose in a way that makes it non-trivial to check, and we advise the reader to do so.  These examples correspond (in order) to the three cases of Corollary~\ref{cor:covers}:
\begin{displaymath}\begin{aligned}
	& (2\;6\;3)(4)(7\;12\;5\;8\;11)(1)(9)(10) && \lessdot_\bth && (5\;8\;11\;4\;6\;3\;2\;7\;12)(1)(9)(10);\\% multiply by (2 11 4)
	& (2\;6\;3)(4)(7\;12\;9\;5\;8\;11)(1\;10) && \lessdot_\bth && (5\;8\;11\;4\;6\;3\;2\;7\;12\;9)(1\;10);\\% multiply by (2 11 4)
	%& (5\;8\;11\;4)(6\;3\;2\;7\;12\;9)(1)(10) && \lessdot_\bth && (8\;11\;7\;3\;2\;4\;5)(12\;9\;6)(1)(10).
	& (7\;12\;9\;6\;3\;2\;1\;10\;11)(4\;5\;8) && \lessdot_\bth && (5\;8\;4\;1\;10\;11)(6\;3\;2\;7\;12\;9).% multiply by (2 11 4)
\end{aligned}\end{displaymath}
%For the first case $(2\; 6\; 3)(4)(7\; 12\; 5\; 8 \; 11)\; z \lessdot_\bth ( 5\;8\;11\; 4\; 6\; 3\; 2\; 7\; 12)\;  z$. In this example and the other ones, $z$ is simply a product of cycles such that we obtain permutations. For the second case $(2\; 6\; 3)(4)(7\; 12\; 9 \;5\; 8 \; 11)\; z \lessdot_\bth ( 5\;8\;11\; 4\; 6\; 3\; 2\; 7\; 12\;; 9)\;  z$. An example of the third case is  $( 5\;8\;11\; 4)  (6\; 3\; 2\; 7\; 12\; 9)\; z  \lessdot_\bth ( 8\;11\;7\;3\; 2\;  4\; 5)  (12\; 9\; 6)\; z$. 
\end{example}

\begin{remark}\label{rem:odd_below} 
	An immediate consequence of Corollary~\ref{cor:covers} is that if $y$ only has odd cycles in its cycle decomposition, then the same is true for all elements $x\leq_\bth y$. Indeed, we can show by immediate induction that only covers of type~\eqref{it:oddo} can occur.
\end{remark}

\begin{proposition}\label{prop:interval_decomposition}
	Let $x\in\Alt_{N}$. Let $\zeta_1,\zeta_2,\ldots,\zeta_k$ be the odd cycles of $x$ and let $\xi$ be the product of its even cycles. Then the product map $(\zeta_1,\zeta_2,\ldots,\zeta_k,\xi)\mapsto \zeta_1\zeta_2\cdots\zeta_k\xi$ induces an isomorphism of graded posets:
	\begin{displaymath}
		 [e,\zeta_1]_{\bth}\times  [e,\zeta_2]_{\bth}\times \cdots \times [e,\zeta_k]_{\bth}\times [e,\xi]_{\bth} \cong [e,x]_{\bth}.
	\end{displaymath}
\end{proposition}
\begin{proof}
	For $i\in[k]$ let $S_{i}=\mathrm{supp}(\zeta_{i})$, and let $S=\mathrm{supp}(\xi)$.  Since $x$ is a permutation, the set $\{S_{1},S_{2},\ldots,S_{k},S\}$ is a partition of $[N]$.  Hence, the product map
	\begin{displaymath}
		f\colon\Alt_{S_{1}}\times\Alt_{S_{2}}\times\cdots\times\Alt_{S_{k}}\times\Alt_{S}\to\Alt_{N}, \quad (x_{1},x_{2},\ldots,x_{k},y)\mapsto x_{1}x_{2}\cdots x_{k}y
	\end{displaymath}
	is well defined, injective and order preserving. Here if $X$ is any finite set, $\Sym_{X}$ is the group of permutations of $X$, and $\Alt_{X}$ is the corresponding alternating subgroup.

	Now let $y\leq_\bth x$. We will prove by induction on $\la(x)-\la(y)$ that $y$ is in the image of $f$. If $\la(y)=\la(x)$ then $x=y$ and $f(\zeta_1,\ldots,\zeta_k,\xi)=y$. If $\la(y)<\la(x)$ then pick $x'$ such that $y\lessdot_\bth x'\leqa x$; by induction $f(x'_1,x'_2,\ldots,x'_k,y')=x'$ where $x'_i\leq_\bth \zeta_i$ for $i\in[k]$ and $y'\leqa \xi$. By Remark \ref{rem:odd_below}, all the $x'_i$ have only odd cycles. Therefore the possible even cycles of $x'$ are in $\xi$. It follows then from Corollary \ref{cor:covers}  that $y\lessdot_\bth x'$ is also in the image of $f$, so $f$ is surjective. The same analysis shows also that $f^{-1}$ preserves the ordering.
\end{proof}

\begin{remark}
	It should be clear from Corollary~\ref{cor:covers} that we cannot hope for a simple decomposition for products of even cycles.  If we split the even cycles of $y\in\Alt_{N}$ into groups of even sizes (so that we remain in $\Alt_{N}$), then the direct product of the resulting induced intervals is not isomorphic to the original one.  A little thought shows that it can be identified (via the same product map as above) as a proper subposet  of $[e,y]_{\bth}$.

 Consider for instance $y=(1\;2)(3\;4)(5\;6)(7\;8)\in\Alt_{8}$, and let $x_{1}=(1\;2)(3\;4)$ and $x_{2}=(5\;6)(7\;8)$.  The interval $[e,y]_{\bth}$ has $296$ elements, while the two (isomorphic) intervals $[e,x_{1}]_{\bth}$ and $[e,x_{2}]_{\bth}$ each have $10$ elements.  In particular $[e,y]_{\bth}\not\cong[e,x_{1}]_{\bth}\times[e,x_{2}]_{\bth}$.  
\end{remark}

%%%%%%%%%%%%%%%%%%%%%%
\section{Noncrossing Partitions}
	\label{sec:alternating_noncrossing}
%%%%%%%%%%%%%%%%%%%%%%

Let $c=(1\;2\;\ldots\;N)$ be a long cycle of $\Sym_{N}$.  It follows from \cite{biane97some} that the interval $[\id,c]_{\btw}$ in $(\Sym_{N},\leqt)$ is isomorphic to the lattice of noncrossing set partitions of $[N]$ defined in \cite{kreweras72sur}: these are the set partitions $\Pi$ of $[N]$ such that given any two distinct blocks $B_1,B_2$ of $\Pi$, there are no indices $i,j$ in $B_1$ and $k,l$ in $B_2$ satisfying $i<k<j<l$. If each block of $\Pi$ is ordered increasingly and considered as a cycle by extending this order cyclically, then by \cite{biane97some}*{Theorem~1.1} one obtains a permutation $x\leqt c$, and each such permutation can be uniquely obtained in this way.

By a slight abuse of vocabulary, we shall call a permutation $x\in\Sym_{N}$ with $x\leqt c$ a \defn{noncrossing partition}.  Let us write $\nc_{N}$ for the set of all noncrossing partitions of $\Sym_{N}$; and let $\pnc_{N}=(\nc_{N},\leqt)$.  

Let us now consider the set $\Alto_{N}\subseteq\Alt_{N}$ of permutations having only odd cycles. It is clear from Proposition~\ref{prop:interval_decomposition} and the remark following it that this subset plays a special role for the order $\leqa$. In this section we will investigate the structure of intervals in the poset $(\Alto_{N},\leqa)$. In Section~\ref{sec:hurwitz} we will also prove that this subset occurs naturally in the context of the Hurwitz action. 

\begin{theorem}\label{thm:AAo_embedding}
	For $N\geq 3$, and $x\in\Alto_{N}$, the interval $[e,x]_{\bth}$ is an induced subposet of the interval $[e,x]_{\btw}$ in $(\Sym_{N},\leqt)$.
\end{theorem}

The aim of this section is to prove Theorem~\ref{thm:AAo_embedding}.  We start with a few observations. The first follows from Equation~\eqref{eq:reflection_length}, Corollary~\ref{cor:covers}\eqref{it:oddo}, as well as Propositions~\ref{prop:alternating_length} and ~\ref{prop:interval_decomposition}.

\begin{proposition}\label{prop:AAo_first_properties}
	$\Alto_{N}$ is a lower ideal of $(\Alt_N,\leq_A)$ and $\la(x)=\tfrac{\lt(x)}{2}$ for any $x\in \Alto_{N}$. \\
Moreover $x\lessdot_A y$ in $\Alto_{N}$ if and only $x$ is obtained from $y$ by splitting an odd cycle of $y$ into three odd cycles of $x$.
\end{proposition}

Next we show that $(\Alto_N,\leqa)$ is a subposet of $(\Sym_N,\leqt)$, in the sense that inclusion is a poset morphism.

\begin{proposition}\label{prop:AAo_subposet}
	For any $x,y\in\Alto_{N}$ we have that $x\leqa y$ implies $x\leqt y$. 
\end{proposition}
\begin{proof}
	Let $x,y\in\Alto_{N}$ with $x\leqa y$.  Since $C_{3,N}$ is closed under $\Alt_{N}$-conjugation Proposition~\ref{prop:generalized_kreweras} implies $x^{-1}y\leqa y$.  By definition we have 
	\begin{displaymath}
		\la(y)=\la(x)+\la(x^{-1}y).
	\end{displaymath}
	All elements belong to $\Alto_{N}$, so we can multiply this equation by two and use Proposition~\ref{prop:AAo_first_properties} to obtain
	\begin{displaymath}
		\lt(y) = \lt(x)+\lt(x^{-1}y),
	\end{displaymath}
	which means precisely that $x\leqt y$.
\end{proof}

\begin{remark}
	The reverse implication does not hold in general.  We have for instance $(2\;4\;5)\leqt (1\;2\;3\;4\;5)$ but $(2\;4\;5)\not\leqa (1\;2\;3\;4\;5)$. Theorem~\ref{thm:AAo_embedding} says that it holds precisely when $x$ and $y$ are both smaller than a given element of $\Alto_N$.
\end{remark}
 
Since any two elements in $\Alto_{N}$ that have the same cycle type are $\Sym_{N}$-conjugate, Proposition~\ref{prop:invariant_conjugation} states that they induce isomorphic intervals in $(\Alto_{N},\la)$.  Moreover, if $y\in\Alto_{N}$ has cycle type $\lambda_y=(2k_1+1,2k_2+1,\ldots,2k_r+1)$, then Proposition~\ref{prop:interval_decomposition} states that 
\begin{displaymath}
	[\id,y]_{\bth} \cong \prod_{i=1}^{r}\bigl[\id,(1\;2\;\ldots\;2k_i\!+\!1)\bigr]_{\bth}.
\end{displaymath}
If we want to understand the intervals in $(\Alto_{N},\leq_{\bth})$ it is therefore sufficient to consider those intervals induced by a single cycle, and we might as well pick our favorite one.  Clearly, the cycle $c=(1\;2\;\ldots\;N)$ belongs to $\Alto_{N}$ if and only if $N$ is odd.  We now define a particular subset of $\nc_{N}$ in terms of the following property for an increasing cycle $(u_{1}<u_{2}<\cdots<u_{q})$:
\begin{description}
	\item[\namedlabel{it:od}{OD}] $q$ is odd, and $u_{j+1}-u_{j}$ is odd for all $j\in[q-1]$.
\end{description}
We say that $x\in\nc_{N}$ satisfies Property~\eqref{it:od} if all its cycles do.

\begin{definition}\label{def:even_nc}
	For $N\geq 3$, we define $\enc_{N}$ to be the set of all $x\in\nc_{N}$ which satisfy Property~\eqref{it:od}.
\end{definition}

\begin{figure}
	\centering
	\begin{tikzpicture}[scale=.75]%[sideways]\small
		\def\x{1.15};
		\def\y{3};
		\def\s{.45};
		\draw(7.5*\x,1.25*\y) node[scale=\s](n1){$(1)$};
		\draw(1*\x,2*\y) node[scale=\s](n2){$(1\;6\;7)$};
		\draw(2*\x,2*\y) node[scale=\s](n3){$(2\;3\;6)$};
		\draw(3*\x,2*\y) node[scale=\s](n4){$(4\;5\;6)$};
		\draw(4*\x,2*\y) node[scale=\s](n5){$(1\;4\;7)$};
		\draw(5*\x,2*\y) node[scale=\s](n6){$(2\;3\;4)$};
		\draw(6*\x,2*\y) node[scale=\s](n7){$(2\;5\;6)$};
		\draw(7*\x,2*\y) node[scale=\s](n8){$(1\;2\;7)$};
		\draw(8*\x,2*\y) node[scale=\s](n9){$(3\;4\;7)$};
		\draw(9*\x,2*\y) node[scale=\s](n10){$(5\;6\;7)$};
		\draw(10*\x,2*\y) node[scale=\s](n11){$(1\;2\;5)$};
		\draw(11*\x,2*\y) node[scale=\s](n12){$(3\;4\;5)$};
		\draw(12*\x,2*\y) node[scale=\s](n13){$(3\;6\;7)$};
		\draw(13*\x,2*\y) node[scale=\s](n14){$(1\;2\;3)$};
		\draw(14*\x,2*\y) node[scale=\s](n15){$(1\;4\;5)$};
		\draw(1*\x,3*\y) node[scale=\s](n16){$(1\;6\;7)(3\;4\;5)$};
		\draw(2*\x,3*\y) node[scale=\s](n17){$(1\;2\;3\;6\;7)$};
		\draw(3*\x,3*\y) node[scale=\s](n18){$(1\;2\;3)(4\;5\;6)$};
		\draw(4*\x,3*\y) node[scale=\s](n19){$(1\;4\;5\;6\;7)$};
		\draw(5*\x,3*\y) node[scale=\s](n20){$(1\;6\;7)(2\;3\;4)$};
		\draw(6*\x,3*\y) node[scale=\s](n21){$(2\;3\;4\;5\;6)$};
		\draw(7*\x,3*\y) node[scale=\s](n22){$(1\;2\;7)(4\;5\;6)$};
		\draw(8*\x,3*\y) node[scale=\s](n23){$(1\;2\;3\;4\;7)$};
		\draw(9*\x,3*\y) node[scale=\s](n24){$(2\;3\;4)(5\;6\;7)$};
		\draw(10*\x,3*\y) node[scale=\s](n25){$(1\;2\;5\;6\;7)$};
		\draw(11*\x,3*\y) node[scale=\s](n26){$(1\;2\;7)(3\;4\;5)$};
		\draw(12*\x,3*\y) node[scale=\s](n27){$(3\;4\;5\;6\;7)$};
		\draw(13*\x,3*\y) node[scale=\s](n28){$(1\;2\;3)(5\;6\;7)$};
		\draw(14*\x,3*\y) node[scale=\s](n29){$(1\;2\;3\;4\;5)$};
		\draw(7.5*\x,3.75*\y) node[scale=\s](n30){$(1\;2\;3\;4\;5\;6\;7)$};
		\draw(n1) -- (n2);
		\draw(n1) -- (n3);
		\draw(n1) -- (n4);
		\draw(n1) -- (n5);
		\draw(n1) -- (n6);
		\draw(n1) -- (n7);
		\draw(n1) -- (n8);
		\draw(n1) -- (n9);
		\draw(n1) -- (n10);
		\draw(n1) -- (n11);
		\draw(n1) -- (n12);
		\draw(n1) -- (n13);
		\draw(n1) -- (n14);
		\draw(n1) -- (n15);
		\draw(n2) -- (n16);
		\draw(n2) -- (n17);
		\draw(n2) -- (n19);
		\draw(n2) -- (n20);
		\draw(n2) -- (n25);
		\draw(n3) -- (n17);
		\draw(n3) -- (n21);
		\draw(n4) -- (n18);
		\draw(n4) -- (n19);
		\draw(n4) -- (n21);
		\draw(n4) -- (n22);
		\draw(n4) -- (n27);
		\draw(n5) -- (n19);
		\draw(n5) -- (n23);
		\draw(n6) -- (n20);
		\draw(n6) -- (n21);
		\draw(n6) -- (n23);
		\draw(n6) -- (n24);
		\draw(n6) -- (n29);
		\draw(n7) -- (n21);
		\draw(n7) -- (n25);
		\draw(n8) -- (n17);
		\draw(n8) -- (n22);
		\draw(n8) -- (n23);
		\draw(n8) -- (n25);
		\draw(n8) -- (n26);
		\draw(n9) -- (n23);
		\draw(n9) -- (n27);
		\draw(n10) -- (n19);
		\draw(n10) -- (n24);
		\draw(n10) -- (n25);
		\draw(n10) -- (n27);
		\draw(n10) -- (n28);
		\draw(n11) -- (n25);
		\draw(n11) -- (n29);
		\draw(n12) -- (n16);
		\draw(n12) -- (n21);
		\draw(n12) -- (n26);
		\draw(n12) -- (n27);
		\draw(n12) -- (n29);
		\draw(n13) -- (n17);
		\draw(n13) -- (n27);
		\draw(n14) -- (n17);
		\draw(n14) -- (n18);
		\draw(n14) -- (n23);
		\draw(n14) -- (n28);
		\draw(n14) -- (n29);
		\draw(n15) -- (n19);
		\draw(n15) -- (n29);
		\draw(n16) -- (n30);
		\draw(n17) -- (n30);
		\draw(n18) -- (n30);
		\draw(n19) -- (n30);
		\draw(n20) -- (n30);
		\draw(n21) -- (n30);
		\draw(n22) -- (n30);
		\draw(n23) -- (n30);
		\draw(n24) -- (n30);
		\draw(n25) -- (n30);
		\draw(n26) -- (n30);
		\draw(n27) -- (n30);
		\draw(n28) -- (n30);
		\draw(n29) -- (n30);
	\end{tikzpicture}
	\caption{The poset $\penc_{7}=(\enc_{7},\leqa)$.}
	\label{fig:enc_3_poset}
\end{figure}

Figure~\ref{fig:enc_3_poset} shows the poset $\penc_{7}=(\enc_{7},\leqa)$.  The next result states that when $N$ is odd, the set $\enc_{N}$ in fact induces an interval in $(\Alt_{N},\leqa)$. 
	
\begin{proposition}\label{prop:onc_characterization}
	A permutation $x\in\Alt_{2n+1}$ satisfies $x\leqa (1\;2\;\ldots\;2n+1)$ if and only if $x\in\enc_{2n+1}$.
\end{proposition}
\begin{proof}
	Let $c=(1\;2\;\ldots\;2n\!+\!1)$.  Suppose $x\leqa c$. All cycles of $x$ are odd by Proposition~\ref{prop:AAo_first_properties}. From Proposition~\ref{prop:AAo_subposet}, we have $x\leqt c$ so that $x\in\nc_{2n+1}$, which implies that all cycles of $x$ are increasing and noncrossing; see~\cite{brady01partial}*{Theorem~2.14} or \cite{biane97some}.
	
	We prove that $x$ satisfies \eqref{it:od} by induction on $f(x)=n-\la(x)$. If $f(x)=0$, then $x=c$ which clearly satisfies \eqref{it:od}. If $f(x)>0$, then there exists $y\in \enc_{2n+1}$ with $x\lessdot_{\bth} y \leqa c$, so $f(y)=f(x)-1$ and we can assume by induction that the cycles of $y$ satisfy \eqref{it:od}. By Proposition~\ref{prop:AAo_first_properties}, $x$ is obtained from $y$ by splitting an odd cycle $(u_{1}<u_{2}<\cdots<u_{2q+1})$ into three odd cycles, and we need to verify that each such cycle satisfies \eqref{it:od}. This is clear if $u_1$ and $u_{2q+1}$ are in different cycles of $x$. If they are in the same cycle, say in
	\begin{displaymath}
		\zeta=(u_{1}<u_{2}<\cdots<u_{k}<u_{l}<u_{l+1}<\cdots<u_{2q+1})
	\end{displaymath}
	for some $l>k+1$, then we need to show that $u_l-u_k$ is odd, which is equivalent to $l-k$ being odd. This is indeed the case since the two other cycles are odd and thus the set $\{u_{k+1},u_{k+2},\ldots,u_{l-1}\}$ has even cardinality.
	\smallskip

	Conversely, assume that $x\in\nc_{2n+1}$ and its cycles satisfy \eqref{it:od}. If $x=c$ we are done. We now assume $x\neq c$. We will construct $y\in \Alt_{N}$ such that $x\lessdot_{\bth} y\leqt c$ and all cycles of $y$ satisfy \eqref{it:od}. Then $x\leqa c$ by induction on $f(x)$ as above.

	Consider the cycle $\zeta_1=(u_1<u_2<\cdots<u_{2q+1})$ of $x$ with $u_{1}=1$, and set $u_{2q+2}=2n+2$.  Since $x\neq c$, there exists $t\in[2q+1]$ such that $u_{t+1}-u_t>1$. Let $\zeta_2$ be the cycle containing $u_t+1$, say $(v_1<v_2<\cdots<v_{2r+1})$ with $v_{1}=u_{t}+1$, which because of the noncrossing condition satisfies $v_{2r+1}<u_{t+1}$. Now we obtain
	\begin{displaymath}
		u_{t+1}\equiv u_t+1=v_1\equiv v_{2r+1} \pmod{2},
	\end{displaymath}
	and we can thus consider a third cycle $\zeta_3=(w_1<w_2<\cdots<w_{2s+1})$ with $w_{1}=v_{2r+1}+1$, where $w_{2s+1}<u_t+1$ still by the noncrossing condition, and $u_t+1-w_{2s+1}$ is odd. Now set $i=u_t, j=v_{2r+1}, k=w_{2s+1}$. Then $\zeta_1\zeta_2\zeta_3\cdot (i\;j\;k)$ is equal to 
	\begin{displaymath}
	(u_1<\cdots<u_t<v_1<\cdots<v_{2r+1}<w_1<\cdots<w_{2s+1}<u_t+1<\cdots<u_{2q+1})
	\end{displaymath}
	which satisfies \eqref{it:od}. Let $y=x\cdot (i\;j\;k)$; by construction $y\in\nc_{2n+1}$, $x\lessdot_{\bth} y$ and all its cycles satisfy \eqref{it:od}.
\end{proof}

In Remark \ref{rem:altproof} we shall give an alternative proof of this result based on a bijection due to Goulden and Jackson~\cite{goulden00transitive}.

\begin{lemma}\label{lem:value_characterization}
	Let $x\in \nc_{N}$. Then  $x\in\enc_{N}$  if and only if the following property holds:
	\begin{description}
		\item[\namedlabel{it:p}{P}] for any $j\in [N]$, we have $j<x(j)$ if and only if $x(j)-j$ is odd.
	\end{description}
\end{lemma}
\begin{proof}
	Let $j\in[N]$.  If $x(j)=j$, then there is nothing to show.  If $x(j)\neq j$, then $x$ must have a cycle in which $j$ and $x(j)$ are consecutive.  Since $x\in\nc_{N}$, we can write each cycle such that its entries are strictly increasing.  Then, $x(j)<j$ implies that the cycle in question has $x(j)$ as its least element, and $j$ as its greatest.  
	
	If $x$ satisfies \eqref{it:od}, then neighboring entries in a cycle have odd difference, and the extremal entries have even difference, and \eqref{it:p} follows.  Conversely, if $x$ satisfies \eqref{it:p}, then it follows that neighboring entries in a cycle have odd difference, and this combined with the fact that extremal entries in a cycle have even difference implies that all cycles are odd.  Hence $x$ satisfies \eqref{it:od}.  
\end{proof}

Let $y\in\nc_{N}$, and recall from Proposition~\ref{prop:generalized_kreweras} that the Kreweras complement $K_y\colon[\id,y]_{\btw}\to [\id,y]_{\btw}$ is defined by $x\mapsto x^{-1}y$.  

\begin{proposition}\label{prop:kreweras_invariance}
	Suppose $x,y\in\enc_{N}$ and $x\leqt y$. Then $K_y(x)\in\enc_{N}$.
\end{proposition}
\begin{proof}
	Assume first $y=c_k=(1\;2\;\ldots\;2k\!+\!1)$ for a certain $k$, and write $x'=K_{c_k}(x)$. Then $x'(i)=x^{-1}(i+1)$ for $i<2k+1$ and $x'(2k+1)=x^{-1}(1)$. By Lemma~\ref{lem:value_characterization}, we assume that $x$ satisfies \eqref{it:p} and we must show that so does $x'$.

	Suppose $i<x'(i)$. In particular we have $i\neq 2k+1$ so that we have $x'(i)=x^{-1}(i+1)=:j$.  If $i\leq j-2$, then we have $x(j)=i+1<j$, and \eqref{it:p} implies that $j-x(j)$ is even which is equivalent to $i-x'(i)$ being odd as desired.  If $i=j-1$, then $x'(i)=i+1$, and we get $x'(i)-i=1$, which is odd as desired.

	Suppose $i\geq x'(i)$. If $i<2k+1$, then we have $x'(i)=x^{-1}(i+1)=:j$.  We have $x(j)=i+1>j$, and \eqref{it:p} implies that $x(j)-j$ is odd which is equivalent to $i-x'(i)$ being even as desired.  If $i=2k+1$, then we have $x'(i)=x^{-1}(1)$, and $2k+1-x^{-1}(1)$ is even if and only if $x^{-1}(1)$ is odd.  This is clearly satisfied if $x^{-1}(1)=1$.  Otherwise, if $x^{-1}(1)>1$, then \eqref{it:p} applied to $x$ and $x^{-1}(1)$ yields that $x^{-1}(1)-1$ is even as desired.

	We have thus shown the claim for $y=c_k$. If $y$ is a cycle $\zeta=(a_1\;a_2\;\ldots\;a_{2k+1})$ satisfying \eqref{it:od}, then by Proposition~\ref{prop:invariant_conjugation} the  posets $[\id,y]_{\bth}$ and $(1\;2\;\ldots\;2k+1)$ are isomorphic via conjugation by $w\in\Sym_{2n+1}$ satisfying $w(i)=a_i$. Since $a_i-a_j$ has the same parity as $i-j$, the property \eqref{it:od} is preserved by this conjugation, and we can conclude the claim for $y=\zeta$ since $K_{c_k}=wK_\zeta w^{-1}$.
		
	In the general case, one has a poset isomorphism $[\id,y]_{\bth}\cong [\id,\zeta_1]_{\bth}\times[\id,\zeta_{2}]_{\bth}\cdots\times [\id,\zeta_k]_{\bth}$ where the $\zeta_i$ are the odd cycles of $y$, thanks to Proposition~\ref{prop:interval_decomposition}, and each $\zeta_i$ satisfies  \eqref{it:od} because $y$ does. This isomorphism sends $K_y$  to the product of the $K_{\zeta_i}$, and obviously $x\leqa y$ satisfies \eqref{it:od} if and only if each $x_i$ in its image $(x_1,x_2,\ldots,x_k)$ satisfies \eqref{it:od}, which concludes the proof.
\end{proof}

\begin{theorem}\label{thm:enc_subposet}
	For $N\geq 3$ the poset $\penc_{N}$ is an induced subposet of $\pnc_{N}$.
\end{theorem}
\begin{proof}
	We know that $\penc_{N}$ is a subposet of $\pnc_{N}$ by Proposition~\ref{prop:AAo_subposet}. Now let $x,y\in\enc_{N}$ with $x\leqt y$. By Proposition~\ref{prop:kreweras_invariance}, $K_y(x)=x^{-1}y$ is in $\enc_{N}\subset \Alto_{N}$, so we can use the relation between $\lt$ and $\la$ in Proposition~\ref{prop:AAo_first_properties} to deduce that $x\leqa y$, which concludes the proof.
\end{proof}

Now we can conclude the proof of Theorem~\ref{thm:AAo_embedding}.

\begin{proof}[Proof of Theorem~\ref{thm:AAo_embedding}]
	Recall that any $y\in\Alto_{N}$ is $\Sym_{N}$-conjugate to some $x\in\enc_{N}$.  Therefore, we can find $w\in\Sym_{N}$ with $y=w^{-1}xw$.  Proposition~\ref{prop:invariant_conjugation} states that the posets $[\id,x]_{\bth}$ and $[\id,y]_{\bth}$ are isomorphic, and so are $[\id,x]_{\btw}$ and $[\id,y]_{\btw}$.  Now Theorem~\ref{thm:enc_subposet} states that $[\id,x]_{\bth}$ is an induced subposet of $[\id,x]_{\btw}$, and this property is certainly preserved under the above isomorphism. 
\end{proof}

\section{Enumerative Results }
	\label{sec:alternating_combinatorics}
In this section we collect a few enumerative properties of the poset $(\Alt_{N},\leqa)$.
	
\subsection{The Rank Generating Function of $(\Alt_N,\la)$}

Let $F_N(q)=\sum_{x\in \Alt_N}q^{\oc(x)}$ be the polynomial enumerating $\Alt_{N}$ with respect to $\oc$. Then Proposition~\ref{prop:alternating_length} states that $q^{N/2}F_N(q^{-1/2})$ is the polynomial enumerating $\Alt_{N}$ with respect to $\la$.

\begin{proposition}\label{prop:rankenumeration}
	The generating function $F(t,q)=\displaystyle\sum_{N\geq 0}F_N(q)\frac{t^N}{N!}$ is given by
	\begin{align}
	F(t,q) & = \frac{1}{2}\left(\frac{1+t}{1-t}\right)^{q/2}\left(\vphantom{\frac{t+1}{t-1}}(1-t^2)^{-1/2}+(1-t^2)^{1/2}\right) \label{eq11}\\
	& = \frac{1}{2}\left(\vphantom{\frac{t+1}{t-1}}(1+t)^{\frac{q-1}{2}}(1-t)^{-\frac{q+1}{2}}+ (1+t)^{\frac{q+1}{2}}(1-t)^{-\frac{q-1}{2}}\right).\label{eq12}
	\end{align}
\end{proposition}
\begin{proof}
	We need to count collections of cycles with an even number of even cycles, with $q$ marking the number of odd cycles. Recall that the series enumerating cycles is $-\log(1-t)$, so that, by taking its odd and even part, the series for odd cycles is $1/2\log\bigl((1+t)/(1-t)\bigr)$ while the series for even cycles is $-1/2\log(1-t^2)$. By the exponential formula~\cite{stanley01enumerative_vol2}*{Chapter~5}, using $q$ to mark odd cycles and $\bigl(\exp(x)+\exp(-x)\bigr)/2$ to ensure an even number of even cycles, we get
	\begin{displaymath}
		F(t,q) = \exp\left(\frac{q}{2}\log\left(\frac{1+t}{1-t}\right)\right)\cdot\frac{1}{2}\left(\exp\left(\frac{-\log(1-t^2)}{2}\right) + \exp\left(\frac{-\log(1-t^2)}{2}\right)\right)
	\end{displaymath}
	which gives the desired expression~\eqref{eq11}, and ~\eqref{eq12} follows immediately.
\end{proof}

Notice that the case $q=1$ gives $F(t,1)=\bigl(1+t+1/(1-t)\bigr)/2$ which corresponds to the fact that $\lvert\Alt_{0}\rvert=\lvert\Alt_{1}\rvert=1$ and $\lvert\Alt_{N}\rvert=N!/2$ for $N\geq 2$. We obtain $F_N(q)$ by expanding \eqref{eq12} in powers of $t$ and picking the coefficient of $t^N/N!$. There does not seem to be a nice formula for $F_N(q)$ in general, however it is easy to use any mathematics software to obtain the polynomials for the first few values of $n$, see Table~\ref{tab:rank_generating_function}.

\begin{table}
	\centering
	\begin{tabular}{c|c|c}
		$n$ & $F_{n}(q)$ & $\displaystyle \sum_{x\in \Alt_{N}}q^{\la(x)}$\\
		\hline
		$0$ & $1$  & $1$ \\
		$1$ & $q$  & $1$ \\
		$2$ & $q^{2}$ & $q$ \\ 
		$3$ & $q^{3}+2q$  & $1+2q$ \\
		$4$ & $q^{4}+8q^{2}+3$  & $1+8q+3q^2$ \\
		$5$ & $q^{5}+20q^{3}+39q$ & $1+20q+39q^2$ \\
		$6$ & $q^{6}+40q^{4}+229q^{2}+90$  &  $1+40q+229q^2+90q^3$ \\
		$7$ & $q^{7}+70q^{5}+889q^{3}+1560q$  & $1+70q+889q^2+ 1560 q^3$ 
	% 	$8$ & $q^{8}+112q^{6}+2674q^{4}+12648q^{2}+4725$  \\
	% 	$9$ & $q^{9}+168q^{7}+6762q^{5}+67472q^{3}+107037q$  \\
	% 	$10$ & $q^{10}+240q^{8}+15078q^{6}+273860q^{4}+1128321q^{2}+396900$ \\
            \end{tabular}
	\caption{The rank generating function of $(\Alt_{N},\leqa)$ for $N\leq 7$.
	\label{tab:rank_generating_function}}
\end{table}

\subsection{The Poset $\penc_{N}$}
	\label{sec:enumerative}
Recall from Proposition~\ref{prop:interval_decomposition} that any interval in $(\Alt_{N},\leqa)$ can be decomposed as a direct product of intervals induced by odd cycles times one interval induced by an element consisting of an even number of even cycles.  In this section we study the intervals $\penc_{N}$ defined in Definition~\ref{def:even_nc}.  In view of Propositions~\ref{prop:interval_decomposition} and~\ref{prop:AAo_first_properties} this knowledge is enough to understand all intervals in the poset $(\Alto_{N},\leqa)$.  

Given a generated group $(G,\Gen)$, let $y\in G$ and $m\geq 1$. Then the multichains  $y_{1}\leq_\Gen y_{2}\leq_\Gen\cdots \leq_\Gen y_{m}\leq_\Gen y$ are in bijection with the factorizations $x_1x_{2}\cdots x_{m+1}=y$ satisfying  $\ell_{\Gen}(x_1)+\ell_{\Gen}(x_2)+\cdots +\ell_{\Gen}(x_{m+1})=\ell_{\Gen}(y)$ with $x_i\in G$. One simply defines $x_i=y^{-1}_{i-1}y_i$  for $i\in[m+1]$ with $y_0=e$ and $y_{m+1}=y$.  The inverse bijection is given by setting $y_i=x_1x_2\cdots x_i$.

In particular, $m$-multichains of $\penc_{2n+1}$ correspond bijectively to factorizations $x_{1}x_{2}\cdots x_{m+1}=(1\;2\;\ldots\;2n+1)$ in $\Alt_{N}$ such that $\la(x_{1})+\la(x_{2})+\cdots+\la(x_{m+1})=n$.  By Proposition~\ref{prop:kreweras_invariance}, all $x_i$ in this factorization belong to $\penc_{2n+1}$. Thus the $x_{i}$ have only odd cycles and therefore $\lt(x_i)=2\la(x_i)$ by Proposition~\ref{prop:AAo_first_properties}. We have the following lemma.

\begin{lemma}\label{lem:multichains_onc}
	The set of $m$-multichains of $\penc_{2n+1}$ is in bijection with the set of factorizations $x_{1}x_{2}\cdots x_{m+1}=(1\;2\;\ldots\;2n+1)$ such that $\lt(x_{1})+\lt(x_{2})+\cdots+\lt(x_{m+1})=2n$, and all factors $x_{i}$ belong to $\Alto_{2n+1}$.
\end{lemma}

 These multichains can thus be found inside $\pnc_{2n+1}$, which enables us to draw from results on the noncrossing partition lattice.	

We start with a formula for the number of multichains of $\penc_{2n+1}$ with fixed rank jump vector, and then derive the zeta polynomial of $\penc_{2n+1}$, as well as some other enumerative properties, from this.

\begin{theorem}\label{thm:enc_rs_chain_enumeration}
	For $n,q\geq 1$, the number of $(q-1)$-multichains $C=(x_{1},x_{2},\ldots,x_{q-1})$ of $\penc_{2n+1}$ with rank jump vector $r(C)=(r_{1},r_{2},\ldots,r_{q})$ is
	\begin{displaymath}
		(2n+1)^{q-1}\prod_{i=1}^{q}{\frac{1}{2n+1-r_{i}}\binom{2n+1-r_{i}}{r_{i}}}.
	\end{displaymath}
\end{theorem}
\begin{proof}
	By Lemma \ref{lem:multichains_onc} and the discussion preceding it, such a multichain is equivalent to a minimal factorization $y_{1}y_{2}\cdots y_{q}=(1\;2\;\ldots\;2n\!+\!1)$ where $y_i\in\Alto_{2n+1}$ for $i\in[q]$. Now ~\cite{goulden92combinatorial}*{Theorem 3.2} (see also~\cite{krattenthaler10decomposition}*{Theorem 5}) gives a formula if the cycle type of each $y_i$ is fixed: in our case, if $y_i$ has $p_j^{(i)}$ cycles of length $2j+1$ for $j\geq 1$, then the number of these factorizations is given by
	\begin{displaymath}
		(2n+1)^{q-1}\prod_{i=1}^{q}{\frac{1}{2n-2r_{i}+1}\binom{2n-2r_{i}+1}{p_1^{(i)},p_2^{(i)},\ldots}},
	\end{displaymath}
	in which $r_i=\sum_{j}{jp_j^{(i)}}$. To obtain all factorizations, we sum over all sequences $(p_{1}^{(i)},p_{2}^{(i)},\ldots)$ by using~\cite{krattenthaler10decomposition}*{Lemma 4}, and we obtain
	\begin{displaymath}
		(2n+1)^{q-1}\prod_{i=1}^{q}{\frac{1}{2n-2r_{i}+1}\binom{2n-r_{i}}{r_{i}}}
	\end{displaymath}
	for the number of such factorizations.  This formula is equivalent to the formula in the statement.
\end{proof}

The cases $q=2,r(C)=(k,n-k)$ on the one hand, and $q=n+2,r(C)=(0,{1},{1},\ldots,{1},0)$ on the other give the following enumerations.

\begin{corollary}\label{cor:enc_rk_and_chain_enumeration}
	For $n\geq 1$ and $k\in\{0,1,\ldots,n\}$, we have the following formulas:
	\begin{align*}
		\Bigl\lvert\Rank_{\penc_{2n+1}}(k)\Bigr\rvert & = \frac{2n+1}{(2n+1-k)(n+1+k)}\binom{2n+1-k}{k}\binom{n+1+k}{n-k};\\
		\Bigl\lvert\MaxChains\bigl(\penc_{2n+1}\bigr)\Bigr\rvert & = (2n+1)^{n-1}.
	\end{align*}
\end{corollary}

The second result is in fact a special case of~\cite{biane96minimal}*{Theorem~1} and corresponds to sequence \cite{sloane}*{A052750}.  Table~\ref{tab:enc_rank_numbers} lists the rank numbers of $\penc_{2n+1}$ for $n\leq 5$.  

\begin{table}
	\centering
	\begin{tabular}{c|c}
		$n$ & Rank numbers of $\penc_{2n+1}$ \\
		\hline
		$1$ & $(1,1)$ \\
		$2$ & $(1,5,1)$ \\
		$3$ & $(1,14,14,1)$ \\
		$4$ & $(1,30,81,30,1)$ \\
		$5$ & $(1,55,308,308,55,1)$ \\
	\end{tabular}
	\caption{The sequences of rank numbers of $\penc_{2n+1}$ for $n\leq 5$.}
	\label{tab:enc_rank_numbers}
\end{table}

Before we continue to prove Theorem~\ref{thm:main_zeta_polynomial}, we record the following auxiliary result, which is a multi-parameter version of the Rothe-Hagen identity.%  Its proof follows essentially from the results noted in \cite{gould56some}.  

\begin{lemma}\label{lem:multi_rothe_hagen}
	Let $r$ be a positive integer, and fix integers $a_{1},a_{2},\ldots,a_{r},b,n$.  Let $a=a_{1}+a_{2}+\cdots+a_{r}$.  We have
	\begin{displaymath}
		\sum_{n_{1}+n_{2}+\cdots+n_{r}=n}{\prod_{i=1}^{r}{\frac{a_{i}}{a_{i}+bn_{i}}\binom{a_{i}+bn_{i}}{n_{i}}}} = \frac{a}{a+bn}\binom{a+bn}{n}.
	\end{displaymath}
\end{lemma}
\begin{proof}
	The key observation to this proof was already made in Equation~(7) of \cite{gould56some}.  It was noted there that for integers $s,t$ we have the identity of power series in $z$:
	\begin{equation}\label{eq:gould_equation}
		 x^{s}=\sum_{j=0}^{\infty}{\frac{s}{s+tj}\binom{s+tj}{j}z^{j}},
	\end{equation}
	%where $z=\tfrac{x-1}{x^{t}}$, and $\lvert z\rvert<\bigl\lvert\tfrac{(t-1)^{t-1}}{t^{t}}\bigr\rvert$.  
	where $x$ is defined as the power series solution of $x=1+zx^t$. Note that $x$ counts plane $t$-ary trees with respect to their number of internal vertices; also,  \eqref{eq:gould_equation} is actually a direct application of the  Lagrange Inversion Theorem. In the present setting we have $x^{a_{1}}x^{a_{2}}\cdots x^{a_{r}}=x^{a}$.  If we apply \eqref{eq:gould_equation} on both sides, we obtain
	\begin{align*}
		\sum_{n=0}^{\infty}{\frac{a}{a+bn}\binom{a+bn}{n}z^{n}} & = x^{a} = x^{a_{1}}x^{a_{2}}\cdots x^{a_{r}} \\
		& = \prod_{i=1}^{r}{\left(\sum_{n=0}^{\infty}{\frac{a_{i}}{a_{i}+bn}\binom{a_{i}+bn}{n}z^{n}}\right)}\\
		& = \sum_{n=0}^{\infty}{\left(\sum_{n_{1}+n_{2}+\cdots+n_{r}=n}\prod_{i=1}^{r}{\frac{a_{i}}{a_{i}+bn_{i}}\binom{a_{i}+bn_{i}}{n_{i}}}\right)}z^{n}.
	\end{align*}
	The claim then follows by comparing coefficients.
\end{proof}

\begin{proof}[Proof of Theorem~\ref{thm:main_zeta_polynomial}]
	By summing the number of $(q-1)$-multichains over all rank jump vectors given by Theorem~\ref{thm:enc_rs_chain_enumeration}, we get
	\begin{align*}
		\ZetaPol_{\penc_{2n+1}}(q) & = \sum_{r_{1}+r_{2}+\cdots+r_{q}=n}{(2n+1)^{q-1}\prod_{i=1}^{q}{\frac{1}{2n-r_{i}+1}\binom{2n-r_{i}+1}{r_{i}}}}\\
		& = \frac{1}{2n+1}\left(\sum_{r_{1}+r_{2}+\cdots+r_{q}=n}{\prod_{i=1}^{q}{\frac{2n+1}{2n-r_{i}+1}\binom{2n-r_{i}+1}{r_{i}}}}\right)\\
		& \overset{*}{=} \frac{1}{2n+1}\left(\frac{q(2n+1)}{q(2n+1)-n}\binom{q(2n+1)-n}{n}\right)\\
		& = \frac{q}{q(2n+1)-n}\binom{q(2n+1)-n}{n}.
	\end{align*}
	as desired.  The equality marked with a '*' is Lemma~\ref{lem:multi_rothe_hagen} with $r=q$, $a_{1}=a_{2}=\cdots=a_{q}=2n+1$ and $b=-1$.
\end{proof}

By evaluating the previous polynomial at $q=2,3$, and $-1$, respectively, we obtain the following corollary.

\begin{corollary}\label{cor:enc_essentials}
	For $n\geq 1$, we have the following formulas:
	\begin{align}
		\Bigl\lvert\enc_{2n+1}\Bigr\rvert & = \frac{1}{n+1}\binom{3n+1}{n}; \label{eq:enc_card}\\
		\Bigl\lvert\Int\bigl(\penc_{2n+1}\bigr)\Bigr\rvert & = \frac{3}{5n+3}\binom{5n+3}{n};\label{eq:enc_intervals}\\
		(-1)^{n}\mu\bigl(\penc_{2n+1}\bigr) & = \frac{1}{4n+1}\binom{4n+1}{n}.\label{eq:enc_moebius}
	\end{align}
\end{corollary}

The formulas appearing in Corollary~\ref{cor:enc_essentials} correspond to sequences \cite{sloane}*{A006013}, \cite{sloane}*{A118970}, and \cite{sloane}*{A002293}, respectively.

\begin{remark}
	Besides studying enumerative aspects of the poset $(\Alt_{N},\leqa)$ we may as well ask for structural or topological properties.  
	
	On the topological side, it is well known that the order complex of the poset $(\Sym_{N},\leqt)$ is Cohen-Macaulay and thus a wedge of spheres~\cite{athanasiadis08absolute}*{Theorem~1}.  Moreover, it follows from \cite{bjorner80shellable}*{Example~2.9} that every interval in $(\Sym_{N},\leqt)$ is in fact (lexicographically) shellable.  
	
	We have verified by computer that every interval in the poset $(\Alt_{N},\leqa)$ is shellable for $N\leq 7$, which leads us to believe that this holds for all $N$.
\end{remark}

\subsection{Bijections}
	\label{sec:bijective}
In this section we reprove \eqref{eq:enc_card} bijectively, and we use this bijection to determine the size of $\enc_{2n}$.

\begin{proof}[Bijective Proof of \eqref{eq:enc_card}]
	Let $N>0$ be an integer.  We first recall the bijection $\GJbij$ from $\nc_{N}$ to the set of edge-rooted bicolored plane trees with $N$ edges due to Goulden and Jackson~\cite{goulden92combinatorial}*{Theorem~2.1}. 
	
	Consider $x\in\nc_{N}$, and let $y=x^{-1}(1\;2\;\ldots\;N)$.  With each cycle of $x$ (respectively $y$) we associate a white (respectively black) vertex in the tree $\GJbij(x)$. The vertex corresponding to a cycle $(a_1\;a_2\;\ldots\;a_p)$ is adjacent to $p$ edges labeled by $a_1,a_{2},\ldots,a_p$ clockwise.  This creates a planar bicolored tree, and in order to obtain $\GJbij(x)$ we root the tree at the edge labeled by $1$, and delete all labels.  To obtain the inverse bijection, we simply need to reconstruct the labels from the tree: this is done by walking around the tree clockwise, and labeling the edges by $1,2,\ldots,N$ starting from the marked edge in the direction from its white to its black vertex.  Clearly this bijection sends the cycle type of $x$ (respectively $y$) to the degree distribution of white (respectively black) vertices in the tree $\GJbij(x)$. 
	
	By Lemma~\ref{lem:multichains_onc}, applied to the case $m=1$, it follows that $\GJbij$ restricts to a bijection between $\enc_{2n+1}$ and marked bicolored plane trees with $2n+1$ edges where all vertices have odd degree.  By deleting the marked edge, we obtain a pair $(T_{1},T_{2})$ of planar \emph{rooted} trees where each node has an even number of children and the total number of edges in $T_{1}$ and $T_{2}$ is $2n$.  In view of \cite{deutsch02diagonally} the set of such pairs of trees is in bijection with the set of pairs of ternary trees with a total of $n$ internal nodes.  That the cardinality of this set is given by \eqref{eq:enc_card} follows for instance from \cite{graham94concrete}*{p.\;201, Equation~(5.60)}.
%	
%	bijection between $\enc_{2n+1}$ and the set of pairs $T_1,T_2\in \mathcal{T}^{e}$ where $\mathcal{T}^{e}$ is the set of planar \emph{rooted} trees where each node has an even number of children, and the total number of edges in $T_1$ and $T_2$ is $2n$.  In view of \cite{deutsch02diagonally}, the latter set is in bijection with the set of pairs of ternary trees with a total of $n$ internal nodes.  That the cardinality of this set is given by \eqref{eq:enc_card} follows for instance from \cite{graham94concrete}*{p.\;201, Equation~(5.60)}.
\end{proof}

We may use this bijection to prove the following result.

\begin{proposition}\label{prop:onc_cardinality_even}
	For $n\geq 1$, we have 
	\begin{align*}
		\Bigl\lvert\enc_{2n}\Bigr\rvert = \frac{1}{2n+1}\binom{3n}{n}.
	\end{align*}
\end{proposition}
\begin{proof}
	Let $x\in\enc_{2n}$.  Since every cycle of $x$ has odd length, and $x$ lives in $\Alt_{2n}$, we conclude that $x$ has an even number of cycles.  Moreover, the permutation $y=x^{-1}(1\;2\;\ldots\;2n)$ has a unique cycle with even length.  (This is the cycle containing $2n$.)  It follows that $\GJbij(x)$ is a marked bicolored plane tree with $2n$ edges, where all vertices have odd degree, except for a single black vertex.
	
	If we make this black vertex the root, we obtain a planar rooted tree with $2n$ edges, where each node has an even number of children, and this process is clearly bijective.  As before, the set of such trees is in bijection with the set of ternary trees with $n$ internal nodes~\cite{deutsch02diagonally}.  The set of such trees is counted by \eqref{eq:gould_equation} with $s=1$ and $t=3$, and yields precisely the formula in the statement.
%	
%	Actually one can also work with $\enc_{2n}$. Here one shows easily that $\GJbij$ sends such elements to marked bicolored plane trees with $2n$ edges where all vertices have odd degree, except the black vertex adjacent to the marked edge. Indeed this vertex corresponds to the cycle of $y$ containing $1$, and this cycle has even length because $x$ has an even number of cycles. By making this black vertex the root, we obtain a bijection between $\enc_{2n}$ and trees in $\mathcal{T}^{e}$ with $2n$ edges. Their enumeration is well known  (it is an application of~\eqref{eq:gould_equation} for $s=1$ and $t=3$) and we thus get the following proposition.
\end{proof}

\begin{example}\label{ex:enc_8}
	Let us illustrate the bijection  with an example.  Let us consider the permutation $x=(1\;14\;15)(3\;4\;7)(8\;9\;10\;11\;12)\in\Alt_{17}$.  Its Kreweras complement is $y=(1\;2\;7\;12\;13)(4\;5\;6)(15\;16\;17)$.  The corresponding labeled planar bicolored tree is shown in Figure~\ref{fig:enc_8_example_labeled_tree}; and the corresponding pair of ternary trees is shown in Figure~\ref{fig:enc_8_example_rooted_trees}.
\end{example}

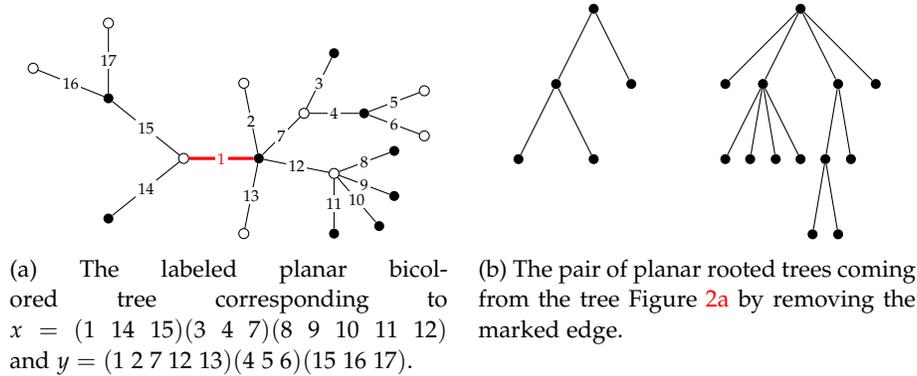
\begin{figure}
	\centering
	\begin{subfigure}[t]{.46\textwidth}
		\centering
		\begin{tikzpicture}
			\def\x{1};
			\def\y{1};
			\def\s{.4};
			\draw(2*\x,2.2*\y) node[fill,circle,scale=\s](m1){};
			\draw(2*\x,3.8*\y) node[fill,circle,scale=\s](m2){};
			\draw(4*\x,3*\y) node[fill,circle,scale=\s](m3){};
			\draw(5*\x,4.4*\y) node[fill,circle,scale=\s](m4){};
			\draw(5.4*\x,3.6*\y) node[fill,circle,scale=\s](m5){};
			\draw(5.8*\x,3.1*\y) node[fill,circle,scale=\s](m6){};
			\draw(5.8*\x,2.5*\y) node[fill,circle,scale=\s](m7){};
			\draw(5.6*\x,2.1*\y) node[fill,circle,scale=\s](m8){};
			\draw(5*\x,2*\y) node[fill,circle,scale=\s](m9){};
			\draw(3*\x,3*\y) node[draw,circle,scale=\s](n1){};
			\draw(1*\x,4.2*\y) node[draw,circle,scale=\s](n2){};
			\draw(2*\x,4.8*\y) node[draw,circle,scale=\s](n3){};
			\draw(3.8*\x,4*\y) node[draw,circle,scale=\s](n4){};
			\draw(4.6*\x,3.6*\y) node[draw,circle,scale=\s](n5){};
			\draw(5*\x,2.8*\y) node[draw,circle,scale=\s](n6){};
			\draw(3.8*\x,2*\y) node[draw,circle,scale=\s](n7){};
			\draw(6.2*\x,3.9*\y) node[draw,circle,scale=\s](n8){};
			\draw(6.2*\x,3.3*\y) node[draw,circle,scale=\s](n9){};
			\draw(m1) -- (n1) node[fill=white,inner sep=.9pt] at (2.5*\x,2.6*\y) {\tiny $14$};
			\draw(m2) -- (n1) node[fill=white,inner sep=.9pt] at (2.5*\x,3.4*\y) {\tiny $15$};
			\draw(m2) -- (n2) node[fill=white,inner sep=.9pt] at (1.5*\x,4*\y) {\tiny $16$};
			\draw(m2) -- (n3) node[fill=white,inner sep=.9pt] at (2*\x,4.3*\y) {\tiny $17$};
			\draw[very thick,red](m3) -- (n1) node[fill=white,inner sep=.9pt] at (3.5*\x,3*\y) {\tiny $1$};
			\draw(m3) -- (n4) node[fill=white,inner sep=.9pt] at (3.9*\x,3.5*\y) {\tiny $2$};
			\draw(m3) -- (n5) node[fill=white,inner sep=.9pt] at (4.3*\x,3.3*\y) {\tiny $7$};
			\draw(m3) -- (n6) node[fill=white,inner sep=.9pt] at (4.5*\x,2.9*\y) {\tiny $12$};
			\draw(m3) -- (n7) node[fill=white,inner sep=.9pt] at (3.9*\x,2.5*\y) {\tiny $13$};
			\draw(m4) -- (n5) node[fill=white,inner sep=.9pt] at (4.8*\x,4*\y) {\tiny $3$};
			\draw(m5) -- (n5) node[fill=white,inner sep=.9pt] at (5*\x,3.6*\y) {\tiny $4$};
			\draw(m5) -- (n8) node[fill=white,inner sep=.9pt] at (5.8*\x,3.75*\y) {\tiny $5$};
			\draw(m5) -- (n9) node[fill=white,inner sep=.9pt] at (5.8*\x,3.45*\y) {\tiny $6$};
			\draw(m6) -- (n6) node[fill=white,inner sep=.9pt] at (5.4*\x,2.95*\y) {\tiny $8$};
			\draw(m7) -- (n6) node[fill=white,inner sep=.9pt] at (5.4*\x,2.65*\y) {\tiny $9$};
			\draw(m8) -- (n6) node[fill=white,inner sep=.9pt] at (5.3*\x,2.45*\y) {\tiny $10$};
			\draw(m9) -- (n6) node[fill=white,inner sep=.9pt] at (5*\x,2.4*\y) {\tiny $11$};
		\end{tikzpicture}
		\caption{The labeled planar bicolored tree corresponding to $x=(1\;14\;15)(3\;4\;7)(8\;9\;10\;11\;12)$ and $y=(1\;2\;7\;12\;13)(4\;5\;6)(15\;16\;17)$.}
		\label{fig:enc_8_example_labeled_tree}
	\end{subfigure}
	\hspace*{.2cm}
	\begin{subfigure}[t]{.46\textwidth}
		\centering
		\begin{tikzpicture}
			\def\x{1};
			\def\y{1};
			\def\s{.4};
			\draw(2*\x,4*\y) node[fill,circle,scale=\s](n1){};
			\draw(1.5*\x,3*\y) node[fill,circle,scale=\s](n2){};
			\draw(2.5*\x,3*\y) node[fill,circle,scale=\s](n3){};
			\draw(1*\x,2*\y) node[fill,circle,scale=\s](n4){};
			\draw(2*\x,2*\y) node[fill,circle,scale=\s](n5){};
			\draw(4.75*\x,4*\y) node[fill,circle,scale=\s](m1){};
			\draw(3.75*\x,3*\y) node[fill,circle,scale=\s](m2){};
			\draw(4.25*\x,3*\y) node[fill,circle,scale=\s](m3){};
			\draw(5.25*\x,3*\y) node[fill,circle,scale=\s](m4){};
			\draw(5.75*\x,3*\y) node[fill,circle,scale=\s](m5){};
			\draw(3.75*\x,2*\y) node[fill,circle,scale=\s](m6){};
			\draw(4.08*\x,2*\y) node[fill,circle,scale=\s](m7){};
			\draw(4.42*\x,2*\y) node[fill,circle,scale=\s](m8){};
			\draw(4.75*\x,2*\y) node[fill,circle,scale=\s](m9){};
			\draw(5.08*\x,2*\y) node[fill,circle,scale=\s](m10){};
			\draw(5.42*\x,2*\y) node[fill,circle,scale=\s](m11){};
			\draw(4.91*\x,1*\y) node[fill,circle,scale=\s](m12){};
			\draw(5.25*\x,1*\y) node[fill,circle,scale=\s](m13){};
			\draw(n1) -- (n2);
			\draw(n1) -- (n3);
			\draw(n2) -- (n4);
			\draw(n2) -- (n5);
			\draw(m1) -- (m2);
			\draw(m1) -- (m3);
			\draw(m1) -- (m4);
			\draw(m1) -- (m5);
			\draw(m3) -- (m6);
			\draw(m3) -- (m7);
			\draw(m3) -- (m8);
			\draw(m3) -- (m9);
			\draw(m4) -- (m10);
			\draw(m4) -- (m11);
			\draw(m10) -- (m12);
			\draw(m10) -- (m13);
		\end{tikzpicture}
		\caption{The pair of planar rooted trees coming from the tree Figure~\ref{fig:enc_8_example_labeled_tree} by removing the marked edge.}
		\label{fig:enc_8_example_rooted_trees}
	\end{subfigure}
	\caption{An illustration of the bijection $\GJbij$.}
	\label{fig:enc_8_example}
\end{figure}

\begin{remark}\label{rem:altproof}
	We want to sketch how Goulden--Jackson's bijection gives an alternative way to prove certain key results of Section \ref{sec:alternating_noncrossing}. 

	We start with Proposition~\ref{prop:onc_characterization}, so let $c= (1\;2\;\ldots\;2n\!+\!1)$ and $x\in \nc_{2n+1}$. The proposition states that  $x\in\enc_{2n+1}$ if and only if $x\leqa c$. By Lemma \ref{lem:multichains_onc}, one has $x\leqa c$ if and only if $xy=c$ is a minimal factorization and $x,y$ have only odd cycles, where $y=x^{-1}c=K_c(x)$. Via Goulden--Jackson's bijection, this means that all vertices of the marked bicolored planar tree $\Phi(x)$ have odd degree. It is then easily proved, by induction starting from the leaves of $\Phi(x)$, that all cycles of $x$ and $y$ necessarily satisfy Property~\eqref{it:od}, and thus $x$ belongs to $\enc_{2n+1}$. 

	To prove the reverse implication, consider $x\in\enc_{2n+1}$. It is enough to prove that $y=x^{-1}c$ has odd cycles. Equivalently, one must show that the white vertices in $\Phi(x)$ all have odd degree if the cycles of $x$ satisfy Property~\eqref{it:od}, which is also done by induction.

	This approach has the additional advantage that it proves Proposition~\ref{prop:kreweras_invariance} for $y$ a long cycle $(1\;2\;\ldots\;2k\!+\!1)$, and thus bypasses the use of Lemma \ref{lem:value_characterization}.
\end{remark}

\subsection{The Case of Two Even Cycles}
	\label{sec:enumeration_two_even_cycles}
In Section~\ref{sec:alternating_noncrossing} we have extensively studied the principal order ideals in $(\Alt_{N},\leqa)$ induced by permutations consisting of only odd cycles.  In view of Proposition~\ref{prop:interval_decomposition} it remains to study those principal order ideals induced by even permutations consisting of only even cycles.  

We restrict our attention to even permutations which have precisely two even cycles, say $x=x_{p,q}=(a_{1}\;a_{2}\;\ldots\;a_{2p})(b_{1}\;b_{2}\;\ldots\;b_{2q})$ for $p\geq q\geq 1$.  Then, $x_{p,q}\in\Alt_{2(p+q)}$, and Proposition~\ref{prop:alternating_length} implies $\ell_{\bth}(x_{p,q})=p+q$.  Let $N=2(p+q)$.

We are not able to describe a combinatorial model for the elements $y\in\Alt_{N}$ with $y\leqa x$, though there is reason to believe that noncrossing partitions on an annulus could be involved. We can however count the number of reduced decompositions of $x_{p,q}$.

\begin{proposition}\label{prop:max_chains_two_even}
	  The number of maximal chains in $[\id,x_{p,q}]_{\bth}$ is 
	\begin{displaymath}
		\frac{2(p+q-1)!(2p)^{p}(2q)^{q}}{(p-1)!(q-1)!}.
	\end{displaymath}
\end{proposition}
\begin{proof}
	We have to count the minimal factorizations of $x_{p,q}$ in $3$-cycles; these factorizations are easily seen to be transitive. Therefore the number that we seek is precisely the coefficient $c_3(2p,2q)$ in \cite{goulden00transitive}, and can be computed via Theorem~1.4 in that paper.
\end{proof}

Let us call an element $y$ of $[\id,x_{p,q}]_{\bth}$ \defn{pure} if the support of any cycle of $y$ is included either in $\{a_{1},a_{2},\ldots,a_{2p}\}$ or $\{b_{1},b_{2},\ldots,b_{2q}\}$; such elements are called  \defn{even} if they contain a cycle of even length, and \defn{odd} otherwise.  Finally, $y$ is \defn{mixed} if it is not pure.  

\begin{proposition}
	There  is a bijection between  even pure elements of $[\id,x_{p,q}]_{\bth}$  and odd ones.  Their common cardinality is given by
	\begin{displaymath}
		\binom{3p-1}{p-1}\binom{3q-1}{q-1}.
	\end{displaymath}
\end{proposition}
\begin{proof}[Sketch of the Proof]
	The Kreweras complement $K_{x_{p,q}}$ is the desired bijection:  it is clear that it leaves stable the set of pure elements, and one checks easily that it exchanges the parity. 

	An even pure element can be decomposed uniquely as a product of a permutation with support in $\{a_1,a_2,\ldots,a_{2p}\}$ and a permutation with support in $\{b_1,b_2,\ldots,b_{2q}\}$; we only need to focus on the first permutation. In the same way as Proposition \ref{prop:onc_characterization}, one can show that this permutation is a noncrossing partition of  $\{a_1,a_2,\ldots, a_{2p}\}$ with exactly one even cycle, and where all cycles $(x_1<x_2<\cdots < x_t)$ have the property that $x_{i+1}-x_i$ is odd for all $i$.  Goulden--Jackson's bijection $\Phi$ then encodes these noncrossing partitions as edge-rooted planar trees with $2p$ edges with one black vertex of even degree and all other vertices (white or black) of odd degree. By making the even-degree vertex the root, we obtain a planar rooted tree with $2p$ edges where each node has an even number of children.  As before we use standard results on tree enumeration to show that the number of such trees is $\binom{3p-1}{p-1}$.
\end{proof}

We do not know of a formula for the number $m_{p,q}$ of mixed elements of $[\id,x_{p,q}]_{\bth}$, and consequently neither for the total number $t_{p,q}$ of elements of $[\id,x_{p,q}]_{\bth}$.  Let us, moreover, denote by $\mu_{p,q}$ the M{\"o}bius number of $[\id,x_{p,q}]_{\bth}$, and let $r_{p,q}$ be the corresponding rank-vector.  Table~\ref{tab:even_cycles_numerology} lists a few of those numbers for small values of $p+q$.

\begin{table}
	\centering
	\begin{tabular}{cc|cc|c|c}
		$p$ & $q$ & $m_{p,q}$ & $t_{p,q}$ & $\mu_{p,q}$ & $r_{p,q}$\\
		\hline\hline
		$1$ & $1$ & $8$ & $10$ & $7$ & $(1,8,1)$ \\
		\hline
		$2$ & $1$ & $48$ & $58$ & $-73$ & $(1,28,28,1)$ \\
		\hline
		$3$ & $1$ & $294$ & $350$ & $671$ & $(1,66,216,66,1)$ \\
		$2$ & $2$ & $336$ & $386$ & $863$ & $(1,72,240,72,1)$ \\
		\hline
		$4$ & $1$ & $1824$ & $2154$ & $-6041$ & $(1,128,948,948,128,1)$ \\
		$3$ & $2$ & $2208$ & $2488$ & $-8495$ & $(1,142,1101,1101,142,1)$ \\
		\hline
		$5$ & $1$ & $11440$ & $13442$ & $54264$ & $(1,220,3050,6900,3050,220,1)$ \\
		$4$ & $2$ & $14304$ & $15954$ & $79855$ & $(1,244,3604,8256,3604,244,1)$ \\
		$3$ & $3$ & $15144$ & $16712$ & $87287$ & $(1,252,3771,8664,3771,252,1)$ \\
		\hline
		$6$ & $1$ & $72384$ & $84760$ & $-488543$ & $(1,348,8046,33985,33985,8046,348,1)$ \\
		$5$ & $2$ & $92400$ & $102410$ & $-738948$ & $(1,384,9545,41275,41275,9545,384,1)$ \\
		$4$ & $3$ & $100992$ & $110232$ & $-845023$ & $(1,402,10237,44476,44476,10237,402,1)$ \\
	\end{tabular}
	\caption{The numerology of the intervals $[\id,x_{p,q}]_{\bth}$ for small values of $p+q$.}
	\label{tab:even_cycles_numerology}
\end{table}

%%%%%%%%%%%%%%%%%%%%%%%%%%%%%%%%%%%%%
\section{Hurwitz Action}
	\label{sec:hurwitz}
%%%%%%%%%%%%%%%%%%%%%%%%%%%%%%%%%%%%%
This section is devoted to the enumeration of Hurwitz orbits on the set $\red_{\bth}(x)$ of reduced factorizations of $x\in\Alt_{N}$ into $3$-cycles.

\begin{theorem}[Same statement as Theorem~\ref{thm:main_hurwitz_orbits}]
\label{thm:main_hurwitz_orbits_bis}
	Let $x\in\Alt_{N}$ for $N\geq 3$, and write $2k$ for its number of cycles of even length. 
 The Hurwitz action on $\red_{\bth}(x)$ has $(2k)_{k}=(k+1)(k+2)\cdots(2k)$ orbits.
\end{theorem}

Let us illustrate Theorem~\ref{thm:main_hurwitz_orbits} with a small example.  Consider the element $x=(1\;2)(3\;4)\in\Alt_{4}$.  According to Proposition~\ref{prop:max_chains_two_even}, $x$ has precisely eight reduced factorizations, and we can check that they fall into the following two Hurwitz orbits:
\begin{displaymath}\begin{aligned}
	& (1\;2\;4)(2\;4\;3), && (1\;3\;4)(1\;2\;4), && (1\;3\;2)(1\;3\;4), && (2\;4\;3)(1\;3\;2);\\
	& (2\;3\;4)(1\;4\;2), && (1\;4\;2)(1\;4\;3), && (1\;4\;3)(1\;2\;3), && (1\;2\;3)(2\;3\;4).
\end{aligned}\end{displaymath}

In the course of the proof of Theorem~\ref{thm:main_hurwitz_orbits}, we will in fact precisely describe the orbits. The case $k=0$ emphasizes once again the special role of $\Alto_{N}$.  Let us state this separately.

\begin{proposition}\label{prop:hurwitz_transitive}
	For $N\geq 3$ and $x\in\Alt_{N}$, we have $x\in\Alto_{N}$ if and only if $\Braid_{\la(x)}$ acts transitively on $\red_{\bth}(x)$.
\end{proposition}
	
We start with the direct implication: if $x\in\Alto_{N}$, then the Hurwitz action is transitive on $\red_{\bth}(x)$. 

 \begin{proof}[Proof of Proposition \ref{prop:hurwitz_transitive} ($\Rightarrow$)]
We reason by induction on $\la(x)$. The result is trivial when $x$ is the identity, so we assume that $\la(x)>0$ and that the claim holds for all elements of $\Alto_{N}$ with length smaller than $\la(x)$.

	Suppose first that $x$ has more than one (non-trivial) cycle in its decomposition, say $x=\zeta_1\zeta_2\cdots\zeta_k$ with $k\geq 2$, so that $\la(\zeta_i)<\la(x)$ for all $i$.  By construction, since $x\in\Alto_{N}$ so are its cycles.  Let $\ww\in\red_{\bth}(x)$.  By Proposition~\ref{prop:interval_decomposition}, $\ww$ is a shuffle of $k$ reduced words $\ww_i\in\red_{\bth}(\zeta_{i})$, and moreover the letters involved in distinct $\ww_i$ commute since the corresponding $3$-cycles have disjoint support.  Therefore the Hurwitz action allows us to bring $\ww$ to the form $\ww_1\ww_2\cdots\ww_k$.  Now by induction the Hurwitz action is transitive on $\red_{\bth}(\zeta_i)$ for all $i\in[k]$, and therefore also on $\red_{\bth}(x)$.

	If $x$ is a single cycle, we can assume without loss of generality that $x=c=(1\;2\;\ldots\;2n\!+\!1)$ which has length $\la(c)=n$.  Let us write $u_{i}=(i,i\!+\!1,i\!+\!2)$ for $i\in[N-2]$, and fix the reduced word $\ww_c=u_{1}u_{3}\cdots u_{2n-1}\in\red_{\bth}(c)$. 

	For any $k\in\{0,1,\ldots,n-1\}$ define $\ww_k:=\sigma^{-1}_1\sigma_2^{-1}\cdots\sigma_k^{-1} \ww_c$.  A direct computation shows that 
	\begin{displaymath}
		\ww_k=(1,2k\!+\!2,2k\!+\!3)u_{1}u_{3}\cdots\reallywidehat{u_{2k+1}}\cdots u_{2n-1},
% 		\ww_k=(1,2k\!+\!2,2k\!+\!3)(1,2,3)(3,4,5)\cdots \reallywidehat{(2k\!+\!1,2k\!+\!2,2k\!+\!3)}\cdots (2n\!-\!1,2n,2n\!+\!1).
	\end{displaymath}
	where the hat indicates omission.  Now for $j\in\{0,1,\ldots,n-k-1\}$ define 
	\begin{displaymath}
		\ww_{k,j}:=(1,2k\!+\!2,2k\!+\!2j\!+\!3)\;\xx\;\yy\;\zz,
	\end{displaymath}
	where the words $\xx,\yy,\zz$ are given by
	\begin{align*}
		\xx & = u_{2k+2j+2}u_{2k+2j}\cdots u_{2k+2}\\
		\yy & = u_{1}u_{3}\cdots u_{2k-1}\\
		\zz & = u_{2k+2j+3}u_{2k+2j+5}\cdots u_{2n-1}.
	\end{align*}
	In particular $\ww_{k,0}=\ww_k$.  By induction on $j$ it is verified that
	\begin{displaymath}
		\sigma_1^2\sigma_2\cdots\sigma_{k+j}\ww_{k,j}=\ww_{k,j+1},
	\end{displaymath}
	which implies that all words $\ww_{k,j}$ are Hurwitz-equivalent to $\ww_c$. 

	Furthermore, it is easily shown that $(\sigma_1\sigma_2\cdots\sigma_{n-1})^n$ acts on a word of length $n$ consisting of $3$-cycles by conjugating each letter by $c$.  The characterization in Proposition~\ref{prop:onc_characterization} implies that any $3$-cycle below $c$ is conjugate to a $3$-cycle of the form $(1,{2k\!+\!2j\!+\!2},{2k\!+\!2j\!+\!3})$.  So we proved that for any $3$-cycle $a$ below $c$, there exists a word $\ww_a$ such that $a\ww_a$ is Hurwitz-equivalent to $\ww_c$.

	Now pick any reduced word for $c$, and write it as $a\ww$ where $a\in A$ is its first letter, and let $x=a^{-1}c$ be the element represented by $\ww$. Then $x$ is in $\Alto_{N}$ by Proposition~\ref{prop:AAo_first_properties}, and so by induction the Hurwitz action is transitive on its reduced expressions. In particular $\ww$ is Hurwitz-equivalent to $\ww_{a}$, and so from the previous paragraph $a\ww$ is Hurwitz-equivalent to $\ww_c$, and the direct implication in Proposition~\ref{prop:hurwitz_transitive} is proved.
\end{proof}

\begin{remark}
	The recursive structure of the proof is inspired by \cite{bessis03dual}*{Proposition 1.6.1}. The result is also a special case of  \cite{lando04graphs}*{Theorem~5.4.11}. The proof there is of a geometric nature, using the dictionary between factorizations and ramified coverings of the Riemann sphere.
\end{remark}

We now deal with the case where $x\in\Alt_{N}$ has $2k$ even cycles for $k>0$.  For $\ww\in\red_{\bth}(x)$ we define an equivalence relation $M_\ww$ on the set of even cycles of $x$ as follows: $M_\ww$ is the finest relation such that $\zeta\sim\zeta'$ whenever there exists a letter of $\ww$ whose support intersects the supports of both $\zeta$ and $\zeta'$.

Let us illustrate the relation $M_{\ww}$ with a concrete example.  Consider 
\begin{displaymath}
	x = (1\;5\;4\;7)(2\;9)(3\;12\;8\;6\;10\;15)(11)(13\;18\;14)(16\;17)\in\Alt_{18},
\end{displaymath}
and the reduced factorization
\begin{displaymath}
	\ww = (1\;2\;9)\cdot(1\;7\;9)\cdot(3\;6\;17)\cdot(1\;5\;4)\cdot(6\;17\;16)\cdot(3\;12\;8)\cdot(6\;10\;15)\cdot(13\;18\;14).
\end{displaymath}
The even cycles of $x$ are $\zeta_{1}=(1\;5\;4\;7)$, $\zeta_{2}=(2\;9)$, $\zeta_{3}=(3\;12\;8\;6\;10\;15)$, and $\zeta_{4}=(16\;17)$.  The supports of the letters $(1\;2\;9)$ and $(1\;7\;9)$ of $\ww$ both intersect the supports of $\zeta_{1}$ and $\zeta_{2}$, so we have $\zeta_{1}\sim\zeta_{2}$.  Moreover, the supports of the letters $(3\;6\;17)$ and $(6\;17\;16)$ of $\ww$ both intersect the supports of $\zeta_{3}$ and $\zeta_{4}$, so we have $\zeta_{3}\sim\zeta_{4}$.  Therefore the equivalence classes of  $M_{\ww}$ are given by $\{\zeta_{1},\zeta_{2}\}$ and $\{\zeta_{3},\zeta_{4}\}$.

\begin{lemma}\label{lem:matchings}
	Let  $x\in \Alt_{N}$ and $\ww\in\red_{\bth}(x)$. All equivalence classes of $M_\ww$ consist of two cycles, \ie $M_\ww$ is a (perfect) \emph{matching}. Moreover, $M_\ww$ is invariant under the Hurwitz action on $\red_{\bth}(x)$, \ie for any $i<\la(x)$, we have $M_{\sigma_i^{\pm 1}\ww}=M_\ww$.
\end{lemma}
\begin{proof}
	Consider the maximal chain from the identity to $x$ in $(\Alt_{N},\leqa)$ corresponding to $\ww\in\red_{\bth}(x)$.  By the description of covers in Corollary~\ref{cor:covers}, there must be $k$ occurrences of a cover of type \eqref{it:even}, and each creates a pair of even cycles (from a pair of odd cycles).  If $x_{i}\lessdot_{\bth}x_{i+1}$ is such a cover, then the $3$-cycle $x_{i}^{-1}x_{i+1}$ implies that the newly created pair of even cycles in $x_{i+1}$ is contained in some block of $M_\ww$.  Since the other possible covers of type \eqref{it:oddo} or \eqref{it:odde} merge three cycles (at most one of which is even), they are necessarily labeled by $3$-cycles whose support cannot involve elements from non-paired even cycles. This shows that all classes of $M_\ww$ have size $2$.
 
	By the analysis from the previous paragraph, the support of each $3$-cycle occurring in $\ww$ is included in the support of $\zeta\zeta'$ for a well-defined pair $\{\zeta,\zeta'\}\in M_\ww$, or in the support of an odd cycle of $x$. So letters corresponding to distinct pairs of $M_\ww$ commute, which entails that the Hurwitz action does not modify  $M_\ww$.
\end{proof}

Lemma~\ref{lem:matchings} effectively reduces the Hurwitz action to the case of two even cycles; we thus consider $x_{p,q}=(a_1\;a_2\;\ldots\;a_{2p})(b_1\;b_2\;\ldots\;b_{2q})\in\Alt_N$ as in Section~\ref{sec:enumeration_two_even_cycles}. % with $a_{1}<a_{2}<\cdots<a_{2p}$ and $b_{1}<b_{2}<\cdots<b_{2q}$
The $3$-cycles below $x_{p,q}$ can be divided into two families: \defn{pure generators} which have the form $(a_i\;a_j\;a_k)$ or $(b_i\;b_j\;b_k)$ and \defn{mixed generators} of the form $(a_i\;a_j\;b_k)$ or $(a_i\;b_j\;b_k)$.  With these notations it is easily seen that in a pure generator we necessarily have $i<j<k$ and exactly one element of the sequence $j-i, k-j, i-k$ is even. For mixed generators $j-i$ must be odd in $(a_i\;a_j\;b_k)$, and $k-j$ must be odd in $(a_i\;b_j\;b_k)$.

\begin{definition}\label{def:parity}
	Given $x_{p,q}$ as above, we define the \defn{parity} of a mixed generator $(a_i\;a_j\;b_k)$ or $(a_i\;b_j\;b_k)$ to be the parity of $k-i$. 
\end{definition}

\begin{remark}\label{rem:mixed}
	It is important to notice that this notion of parity is not canonical and depends on a given specification of $x_{p,q}$. If we modify this specification either by swapping the two cycles, or by shifting cyclically the elements of one of the cycles, the notions of odd and even are exchanged. In any case the partition of mixed generators in two classes remains unchanged.
\end{remark}

\begin{lemma}\label{lem:mixed_generators}
	Let $x_{p,q}=(a_1\;a_2\;\ldots\;a_{2p})(b_1\;b_2\;\ldots\;b_{2q})\in\Alt_N$ as above, and let $\ww\in\red_{\bth}(x_{p,q})$. Then $\ww$ contains at least two mixed generators.  Moreover, all mixed generators in $\ww$ have the same parity.
\end{lemma}
\begin{proof} 
	Consider the maximal chain from the identity to $x_{p,q}$ in $(\Alt_{N},\leqa)$ that corresponds to $\ww$. By the description of covers in Corollary~\ref{cor:covers}, this chain contains a unique occurrence of a cover of type~\eqref{it:even}, cover which is necessarily labeled by a mixed generator.  Now one of the two odd cycles joined in this cover contains elements of both cycles of $x_{p,q}$, and so not all generators below it can be pure, which accounts for at least one other mixed generator beside the one coming from the distinguished cover relation above.

	Now let us consider any two such mixed generators. Using the Hurwitz action,  we can assume that these generators occur in the last two positions of $\ww$.  Up to exchanging the two cycles of $x_{p,q}$, we can also assume that the mixed generator occurring in last position in $\ww$ has the form $(a_i\; a_j\; b_k)$; here $j-i$ is necessarily odd as noticed previously.  We shall also assume that $j<i$, the opposite case proceeds entirely similarly.  Define $x'=x_{p,q}(a_i\;a_j\;b_k)^{-1}=x_{p,q}(a_j\; a_i\; b_k)$, which is equal to 
	\begin{displaymath}
		(a_1\;a_2\;\ldots\;a_j\;a_{i+1}\;a_{i+2}\;\ldots\;a_{2p})(a_{j+1}\;a_{j+2}\;\ldots\;a_i\;b_{k+1}\;b_{k+2}\;\ldots b_{2q}\;b_1\;b_2\;\ldots\;b_k).
	\end{displaymath}
	We have $x'\lessdot x_{p,q}$ by definition, and $x'$ is a product of two odd cycles; only the second one of these cycles can be split into three odd cycles when multiplied by a mixed generator.  Denote by $(a_r\; a_s\; b_t)$ (or $(a_r\; b_s\; b_t)$) the mixed generator occurring second to last in $\ww$.   By an elementary but tedious distinction of cases (depending on how the splitting of the second odd cycle occurs when multiplied by this mixed generator), one shows that there always holds $t-r\equiv k-i \pmod 2$. This shows that the two mixed generators have the same parity. 
\end{proof}

We may now conclude the missing part of the proof of Proposition~\ref{prop:hurwitz_transitive}.

\begin{proof}[Proof of Proposition~\ref{prop:hurwitz_transitive} ($\Leftarrow$)]
	Let $x\in\Alt_N\setminus\Alto_{N}$, and let $2k\geq 2$ be its number of even cycles. To show that the Hurwitz action is not transitive, it is enough by Lemma~\ref{lem:matchings} to consider the case $k=1$, and pick $\ww\in\red_{\bth}(x)$.  If we act on $\ww$ by $\tau=\sigma_i^{\pm 1}$, then at most one letter is modified. By Lemma~\ref{lem:mixed_generators}, the resulting letter cannot be a mixed generator of the parity not already occurring in $\ww$. So in a Hurwitz orbit all occurring mixed generators have the same parity. 

	In order to conclude that we have at least two Hurwitz orbits, it suffices to exhibit two reduced factorizations of $x$, one containing even mixed generators, and one containing odd mixed generators.  This is quickly verified by considering the word
	\begin{multline*}
		\xx'=(a_{2}\;a_{3}\;a_{4})(a_{4}\;a_{5}\;a_{6})\cdots(a_{2p-2}\;a_{2p-1}\;a_{2p})\\
			(b_{2}\;b_{3}\;b_{4})(b_{4}\;b_{5}\;b_{6})\cdots(b_{2q-2}\;b_{2q-1}\;b_{2q}),
	\end{multline*}
	which consists only of pure generators. Now define
	\begin{align}
		\xx_{1} & = (a_{1}\;a_{2}\;b_{2})(a_{2}\;b_{2}\;b_{1})\xx' \label{eq:canonical_word_even_1},\quad\text{and}\\
		\xx_{2} & = (a_{2}\;a_{1}\;b_{2})(a_{1}\;b_{2}\;b_{1})\xx' \label{eq:canonical_word_even_2}.
	\end{align}
	We verify easily that both $\xx_{1}$ and $\xx_{2}$ are reduced factorizations of $x_{p,q}$, and the proof is complete.
%
%	To conclude that there are at least two Hurwitz orbits, note that both parities actually occur in the factorizations~\eqref{eq:canonical_word_even_1} and \eqref{eq:canonical_word_even_2} above, and so the two words belong to different Hurwitz orbits. 
\end{proof}

For the proof of Theorem~\ref{thm:main_hurwitz_orbits}, the two factorizations \eqref{eq:canonical_word_even_1} and \eqref{eq:canonical_word_even_2} will play a crucial role.  We have the following lemma.

\begin{lemma}\label{lem:technical}
	All pure generators and all mixed odd generators occur as first letters of words in the Hurwitz orbit of $\xx_{1}$.
\end{lemma}

The proof of this lemma is rather long and technical, so we withhold it for the moment.  Instead, we prove Theorem~\ref{thm:main_hurwitz_orbits}.

\begin{proof}[Proof of Theorem~\ref{thm:main_hurwitz_orbits}]
	In view of Lemma~\ref{lem:matchings} and Lemma~\ref{lem:mixed_generators} we see that for any matching $M$ of the even cycles of $x$, and any choice of parity for each of the $k$ pairs of even cycles, one obtains a Hurwitz-invariant set of factorizations.  This shows that the number of Hurwitz orbits of $\red_{\bth}(x)$ is at least $(2k-1)!!2^k=(2k)_k$. To show Theorem~\ref{thm:main_hurwitz_orbits}, we thus have to show that each of these Hurwitz-invariant sets is in fact a Hurwitz orbit.

	A byproduct of the proof of Lemma~\ref{lem:matchings} is that it is enough to do the case $k=1$ with no odd cycle:  this will transfer to the general case due to the fact that generators with disjoint supports commute, and that the Hurwitz action realizes this commutation for adjacent letters.

%	\medskip
%
%	Let $N=p+q$ and write $x_{p,q}=(a_1\;a_2\;\ldots\;a_{2p})(b_1\;b_2\;\ldots\;b_{2q})\in\Alt_N$. Consider the word
%	\begin{multline*}
%		\xx'=(a_{2}\;a_{3}\;a_{4})(a_{4}\;a_{5}\;a_{6})\cdots(a_{2p-2}\;a_{2p-1}\;a_{2p})\\
%			(b_{2}\;b_{3}\;b_{4})(b_{4}\;b_{5}\;b_{6})\cdots(b_{2q-2}\;b_{2q-1}\;b_{2q}),
%	\end{multline*}
%	which consists only of pure generators. Now define
%	\begin{align}
%		\xx_1 & = (a_{1}\;a_{2}\;b_{2})(a_{2}\;b_{2}\;b_{1})\xx' \label{eq:canonical_word_even_1},\quad\text{and}\\
%		\xx_{2} & = (a_{2}\;a_{1}\;b_{2})(a_{1}\;b_{2}\;b_{1})\xx' \label{eq:canonical_word_even_2}.
%	\end{align}
%	It is quickly verified that $\xx_{1}$ and $\xx_{2}$ are reduced words for $x_{p,q}$, starting respectively with odd and even mixed generators. 

	We thus have to show that any $\ww\in\red_{\bth}(x_{p,q})$ is in the orbit of either $\xx_{1}$ or $\xx_{2}$ given by \eqref{eq:canonical_word_even_1} and \eqref{eq:canonical_word_even_2}, respectively.  
	
%	By Remark~\ref{rem:mixed}, it  suffices to establish that all reduced words for $x_{p,q}$ containing odd mixed generators belong to the Hurwitz orbit of $\xx_{1}$: let $\red^o_{\bth}(x_{p,q})$ denote this set of reduced words.  Analogously to Proposition~\ref{prop:hurwitz_transitive}, we want to reason inductively, and the key technical result is the following.

%	Before proving the lemma, let us see how it allows us to finish the proof of Theorem~\ref{thm:main_hurwitz_orbits}. 

	If we denote by $\red^{o}_{\bth}(x_{p,q})$ the set of reduced factorizations of $x_{p,q}$ that contain odd mixed generators, then in view of Remark~\ref{rem:mixed} it suffices to show that $\red^{o}_{\bth}(x_{p,q})$ is contained in the Hurwitz orbit of $\xx_{1}$.
	
	Let $\ww=a\ww'$ be any element of $\red^o_{\bth}(x_{p,q})$.  By Lemma~\ref{lem:technical}, we know that there exists $a\xx'$ in the Hurwitz orbit of $\xx_1$. So we just have to show that $\xx'$ is in the Hurwitz orbit of $\ww'$.  Both words are certainly reduced factorizations of $a^{-1}x_{p,q}$. We need to distinguish two cases: if $a$ is mixed, then $a^{-1}x_{p,q}$ is a product of two odd cycles, and we know from Proposition~\ref{prop:hurwitz_transitive} that the Hurwitz action has just one orbit so we are done. If $a$ is pure, then $a^{-1}x_{p,q}$ is a product of two odd cycles and two even ones. One of the even cycles of $a^{-1}x_{p,q}$ is unchanged from $x_{p,q}$, and the other even cycle of $a^{-1}x_{p,q}$ is cut from the other cycle of $x_{p,q}$. This new cycle can be written in way that makes  $\xx'$ and $\ww'$ elements of $\red^o_{\bth}(a^{-1}x_{p,q})$ (because $\xx_1$ and $\ww$ are elements of  $\red^o_{\bth}(x_{p,q})$). By induction,  $\xx'$ and $\ww'$ are then Hurwitz equivalent and the proof is complete.
\end{proof}

\begin{remark}
	One can also study the Hurwitz action from a graph-theoretic point of view: for $x\in\Alt_{N}$,  define a graph on the set $\red_\bth(x)$, where there is an edge between two reduced factorizations if and only if one can be obtained from the other by the action of a Hurwitz operator. Such a Hurwitz graph (with vertex set $\red_\btw{(x)}$) was studied for $x$ a long cycle in $\Sym_{N}$ in \cite{adin14maximal}, and more recently in \cite{heller18generalized}.  (In fact, in \cite{heller18generalized} a slightly different graph is considered which happens to coincide with the Hurwitz graph in the case of $\Sym_{N}$.)  In a more general setting this graph appears in \cite{muehle17connectivity}.  
	
	A natural question is to study radius and diameter of these graphs.  While only lower bounds are known for the diameter, there are concrete results known for the radius.  For instance, if $x$ is a long cycle of $\Sym_{n+1}$, then \cite{adin14maximal}*{Theorem~9.3} establishes that the Hurwitz graph on $\red_\btw{(x)}$ has radius $\frac{(n-1)n}{2}$.  If $x$ is a long cycle of $\Sym_{2n+1}$, then we have verified by computer for $n\leq 6$ that the  Hurwitz graph on $\red_\bth{(x)}$ has radius $\frac{(n-1)(n+2)}{2}$, and we conjecture that this holds for all $n$.
	
	If $x=x_{p,q}\in\Alt_{2p+2q}$ is an element with two even cycles of lengths $2p$ and $2q$, then Theorem~\ref{thm:main_hurwitz_orbits} states that the corresponding Hurwitz graph has two connected components (which are mutually isomorphic).  By Proposition~\ref{prop:max_chains_two_even} each connected component has precisely $\frac{(p+q+1)!(2p)^{p}(2q)^{q}}{(p-1)!(q-1)!}$ vertices.  If we restrict our attention to one such component, then we have verified by computer that the radii in the case $p=1$ and $q\in\{1,2,3,4\}$ are $2,4,8,14$.
\end{remark}

We conclude this section with the proof of Lemma~\ref{lem:technical}.

\begin{proof}[Proof of Lemma~\ref{lem:technical}]
	Note that one can omit ``first'' in the  statement: if a letter occurs in the $i\text{th}$ position of a word $\ww$, then it occurs as the first letter of the word $\sigma_1\sigma_2\cdots\sigma_{i-1}\cdot\ww$. 

	Let $N=p+q$.  First of all, a simple computation shows that the braid word
	\begin{equation}\label{eq:shift_word}
		\gamma = \bigl(\sigma_{N-1}^{-1}\sigma_{N-2}^{-1}\cdots\sigma_{1}^{-1}\bigr)^{N}
	\end{equation}
	acts on any word in $\red_{\bth}(x_{p,q})$ by cyclically sending each letter $a_{i}$ to $a_{i+1}$ and each letter $b_{i}$ to $b_{i+1}$.  Another quick computation shows that
	\begin{displaymath}
		y_{1} := (a_{1}\;a_{2}\;b_{2})^{-1}x_{p,q} = (a_{2}\;a_{3}\;\ldots\;a_{2p}\;b_{2}\;b_{3}\;\ldots\;b_{2q}\;b_{1})\in\Alto_{2N}.
	\end{displaymath}
	In view of Proposition~\ref{prop:hurwitz_transitive}, this implies that any $\ww\in\red_{\bth}(x_{p,q})$ starting with $(a_{1}\;a_{2}\;b_{2})$ lies in the Hurwitz orbit of $\xx_{1}$.  Let $C_1$ denote the set of all $3$-cycles below $y_{1}$, which by Proposition \ref{prop:onc_characterization} is
	\begin{align*}
		C_1 & = \Bigl\{(a_{r}\;a_{s}\;a_{t}),(b_{r}\;b_{s}\;b_{t})\mid 1<r<s<t\;\text{and}\;s-r\;\text{odd},t-s\;\text{odd}\Bigr\}\\
		& \kern1cm \cup\Bigl\{(a_{r}\;a_{s}\;b_{t})\mid 1<r<s, 1<t\;\text{and}\;s-r\;\text{odd},t-s\;\text{even}\Bigr\}\\
		& \kern1cm \cup\Bigl\{(a_{r}\;b_{s}\;b_{t})\mid 1<r,1<s<t\;\text{and}\;s-r\;\text{even},t-s\;\text{odd}\Bigr\}\\
		& \kern1cm \cup\Bigl\{(b_{r}\;b_{s}\;b_{1})\mid 1<r<s\;\text{and}\;r\;\text{odd},s\;\text{even}\Bigr\}\\
		& \kern1cm \cup\Bigl\{(a_{r}\;a_{s}\;b_{1})\mid 1<r<s\;\text{and}\;r\;\text{even},s\;\text{odd}\Bigr\}\\
		& \kern1cm \cup\Bigl\{(a_{r}\;b_{s}\;b_{1})\mid 1<r,1<s\;\text{and}\;r,s\;\text{even}\Bigr\}.
	\end{align*}
	All pure generators with no $a_1$ in their support belong to $C_1$. By Hurwitz-transitivity of $\red_{\bth}(y_{1})$ we conclude that any $\ww\in\red^o_{\bth}(x)$ that starts with a pure generator lies in the Hurwitz orbit of $\xx_{1}$.
		
	The odd mixed generators fall into four categories:
	\begin{enumerate}[(I)]
		\item $(a_{r}\;b_{s}\;b_{t})$ for $1\leq r\leq 2p, 1\leq s<t\leq 2q$ and $t-s$ odd, $t-r$ odd;
		\item $(a_{r}\;b_{t}\;b_{s})$ for $1\leq r\leq 2p, 1\leq s<t\leq 2q$ and $t-s$ odd, $s-r$ odd;
		\item $(a_{r}\;a_{s}\;b_{t})$ for $1\leq r<s\leq 2p, 1\leq t\leq 2q$ and $s-r$ odd, $t-r$ odd;
		\item $(a_{s}\;a_{r}\;b_{t})$ for $1\leq r<s\leq 2p, 1\leq t\leq 2q$ and $s-r$ odd, $t-s$ odd.
	\end{enumerate}
	Those of type (I) and (III) (except for those containing $a_{1}$ and $b_{1}$) belong to $C_1$ and as before any $\ww\in\red^o_{\bth}(x)$ starting with one of these lies in the Hurwitz orbit of $\xx_{1}$.  Suitable application of $\gamma$ from \eqref{eq:shift_word} also produces those odd mixed generators of type (I) or (III), which contain $a_{1}$ or $b_{1}$.  It thus remains to consider odd mixed generators of type (II) and (IV).

	Let $\zeta=(a_{r}\;b_{s}\;b_{t})$ be an odd mixed generator of type (II), \ie $s-t$ is odd.  Without loss of generality we may assume that $r=1$.  (We may obtain the general case by suitable application of $\gamma$ from \eqref{eq:shift_word}.)  For arbitrary integers $i,j$ define the braid words
	\begin{align*}
		\omega_{i} & = \sigma_{1}^{-1}\sigma_{2}^{-2}\sigma_{3}\sigma_{4}\cdots\sigma_{i+1}\\
		\tau_{i,k} & = \sigma_{1}^{2}\sigma_{2}\cdots\sigma_{k-1}\sigma_{k}^{-2}\sigma_{k+1}\sigma_{k+2}\cdots\sigma_{i+k-1}.
	\end{align*}
	Fix $k\in\{2,3,\ldots,q\}$ and define
	\begin{equation*}
		\beta_{p,k} = \begin{cases}\omega_{p}, & \text{if}\;k=2,\\ \tau_{p,k}\tau_{p,k-1}\cdots\tau_{p,3}\omega_{p}, & \text{otherwise}.\end{cases}
	\end{equation*}
	We can then verify that 
	\begin{align*}
		\beta_{p,k}\cdot\xx_{1} & = (a_{1}\;b_{2k-1}\;b_{2})(a_{1}\;b_{2k-3}\;b_{2k-2})(a_{1}\;b_{2k-5}\;b_{2k-4})\cdots(a_{1}\;b_{3}\;b_{4})\\
		& \kern1cm (a_{1}\;a_{2}\;b_{2})(b_{1}\;b_{2k-1}\;b_{2k})(a_{2}\;a_{3}\;a_{4})(a_{4}\;a_{5}\;a_{6})\cdots(a_{2p-2}\;a_{2p-1}\;a_{2p})\\
		& \kern1cm (b_{2k}\;b_{2k+1}\;b_{2k+2})(b_{2k+2}\;b_{2k+3}\;b_{2k+4})\cdots(b_{2q-2}\;b_{2q-1}\;b_{2q}).
	\end{align*}
	Finally define for $k\in\{2,3,\ldots,q\}$ and $j\in\{0,1,\ldots,k-2\}$ the braid words
	\begin{align*}
		\mu_{k} & = \sigma_{1}^{2}\sigma_{2}\sigma_{3}\cdots\sigma_{k-2}\sigma_{k-1}^{2},\\
		\alpha_{p,k,j} & = \mu_{k}^{j}\beta_{p,k}.
	\end{align*}
	We obtain
	\begin{align*}
		\alpha_{p,k,j}\cdot\xx_{1} & = (a_{1}\;b_{2k-1}\;b_{2j+2})(b_{2j}\;b_{2j+1}\;b_{2k-1})(b_{2j-2}\;b_{2j-1}\;b_{2k-1})\cdots(b_{2}\;b_{3}\;b_{2k-1})\\
		& \kern1cm (a_{1}\;b_{2k-3}\;b_{2k-2})(a_{1}\;b_{2k-5}\;b_{2k-2})\cdots(a_{1}\;b_{3+2j}\;b_{4+2j})\\
		& \kern1cm (a_{1}\;a_{2}\;b_{2j+2})(b_{1}\;b_{2k-1}\;b_{2k})(a_{2}\;a_{3}\;a_{4})(a_{4}\;a_{5}\;a_{6})\cdots(a_{2p-2}\;a_{2p-1}\;a_{2p})\\
		& \kern1cm (b_{2k}\;b_{2k+1}\;b_{2k+2})(b_{2k+2}\;b_{2k+3}\;b_{2k+4})\cdots(b_{2q-2}\;b_{2q-1}\;b_{2q}),
	\end{align*}
	where the indices of the middle entries in the triples of the second line are supposed to form a decreasing sequence.  In particular, if $j=k-2$, then this line is supposed to be empty.  We conclude that any $\ww\in\red_{\bth}(x_{p,q})$ starting with an odd mixed generator of type (II) lies in the Hurwitz orbit of $\xx_{1}$.

	Let now $\zeta=(a_{s}\;a_{r}\;b_{t})$ be an odd mixed generator of type (IV), \ie $s-r$ is odd.  Without loss of generality assume that $t=2$.  (We may obtain the general case by suitable application of $\gamma$ from \eqref{eq:shift_word}.)  For arbitrary integers $i,j$ define the braid words
	\begin{align*}
		\mu_{i} & = \sigma_{i}^{2}\sigma_{i+1},\\
		\nu_{i} & = \sigma_{3}^{-1}\sigma_{4}^{-1}\cdots\sigma_{i+1}^{-1},\\
		\omega_{i} & = \mu_{i+1}\mu_{i}\cdots\mu_{2},\\
		\xi_{j,i} & = \nu_{j}\nu_{i}\nu_{i-1}\cdots\nu_{2}.
	\end{align*}
	If there is a choice of parameter that makes one of the occurring sequences non-monotone, then we define the corresponding word to be empty.  For instance $\nu_{1}$ is supposed to be empty.  For $0\leq j<k<p$ define yet another braid word
	\begin{displaymath}
		\beta_{k,j} = \sigma_{1}^{-1}\omega_{j}\sigma_{1}^{-1}\sigma_{2}^{-1}\xi_{k,j}.
	\end{displaymath}
	We obtain 
	\begin{align*}
		\beta_{k,j}\cdot\xx_{1} & = (a_{2k+1}\;a_{2j+2}\;b_{2})(a_{1}\;a_{2k+1}\;a_{2k+2})(a_{1}\;a_{2j}\;a_{2j+1})\cdots(a_{1}\;a_{2}\;a_{3})\\
		& \kern1cm (a_{2j+2}\;b_{2}\;b_{1})(a_{2j+2}\;a_{2j+3}\;a_{2j+4})\cdots(a_{2p-2}\;a_{2p-1}\;a_{2p})\\
		& \kern1cm (b_{2}\;b_{3}\;b_{4})(b_{4}\;b_{5}\;b_{6})\cdots(b_{2q-2}\;b_{2q-1}\;b_{2q}),
	\end{align*}
	and we conclude that any $\ww\in\red_{\bth}(x_{p,q})$ starting with an odd mixed generator of type (IV) lies in the Hurwitz orbit of $\xx_{1}$.  This concludes the proof.
\end{proof}

%%%%%%%%%%%%%%%%%%%%%%
\section{Extensions}
	\label{sec:extensions}
%%%%%%%%%%%%%%%%%%%%%%

\subsection{$m$-Divisible Noncrossing Partitions}
	\label{sec:m_divisible}
In the spirit of \cite{armstrong09generalized}, we can define a partial order on the set of multichains of $\penc_{2n+1}$ as follows.  Fix $m\geq 1$, and consider an $m$-multichain $C=(x_{1},x_{2},\ldots,x_{m})$.  The \defn{extended delta sequence} is $\delta_{o}(C)=(d_{0};d_{1},d_{2},\ldots,d_{m})$ of $\penc_{2n+1}$, where $d_{i}=x_{i}^{-1}x_{i+1}$ for $i\in\{0,1,\ldots,m\}$, as well as $x_{0}=\id$ and $x_{m+1}=c=(1\;2\;\ldots\;2n\!+\!1)$.  

We define a partial order on the set of $m$-multichains of $\penc_{2n+1}$ by setting $C\leq C'$ if and only if $d_{i}\geq_{\bth}d'_{i}$ for $i\in[m]$, where $\delta_{o}(C')=(d'_{0};d'_{1},d'_{2},\ldots,d'_{m})$.  Let us denote the resulting poset by $\penc_{2n+1}^{(m)}$. 

\begin{conjecture}
	For $n,m\geq 1$ the number of maximal chains of $\penc_{2n+1}^{(m)}$ is ${m^n(2n+1)^{(n-1)}}$.  
\end{conjecture}

\begin{conjecture}\label{conj:mdiv_zeta}
	For $n,m\geq 1$, the zeta polynomial of $\penc_{2n+1}^{(m)}$ is 
	\begin{multline*}
		\ZetaPol_{\penc_{2n+1}^{(m)}}(q) =\\
			\frac{m(q-1)+1}{\bigl(2m(q-1)+1\bigr)n+m(q-1)+1}\binom{\bigl(2m(q-1)+1\bigr)n+m(q-1)+1}{n}.
	\end{multline*}
\end{conjecture}

In general $\penc_{2n+1}^{(m)}$ is a graded poset with a greatest element, and several minimal elements.  Let $\widehat{\penc}_{2n+1}^{(m)}$ denote the poset that arises from $\penc_{2n+1}^{(m)}$ by adding a unique minimal element, and let $\overline{\penc}_{2n+1}^{(m)}$ denote the poset that arises from $\penc_{2n+1}^{(m)}$ by identifying all minimal elements.  We pose the following conjectures on the M{\"o}bius numbers of these modified posets.

\begin{conjecture}
	For $n,m\geq 1$ the M{\"o}bius number of $\widehat{\penc}_{2n+1}^{(m)}$ is given by 
	\begin{displaymath}
		(-1)^{n-1}\frac{m-1}{m(2n+1)-1}\binom{m(2n+1)-1}{n}.
	\end{displaymath}
\end{conjecture}

\begin{conjecture}
	For $n,m\geq 1$ the M{\"o}bius number of $\overline{\penc}_{2n+1}^{(m)}$ is given by 
	\begin{multline*}
		(-1)^{n}\left(\frac{m}{(m+1)(2n+1)-1}\binom{(m+1)(2n+1)-1}{n}\right.\\
			\left.-\frac{m-1}{m(2n+1)-1}\binom{m(2n+1)-1}{n}\right).
	\end{multline*}
\end{conjecture}

\subsection{Generation by $k$-Cycles}
	\label{sec:kcycles}
A fairly natural extension is to consider $k$-cycles instead of $3$-cycles for $k\geq 4$.  We may then ask which results of this paper can be generalized from $\leq_\bth$ to $\leq_{\mathbf{k}}$? The difficulty starts at the very beginning: to our knowledge, there is no simple way to express the associated length function $\ell_{\mathbf{k}}$.  In the paper \cite{herzog76representation}, the authors manage to find a complicated formula for $\ell_{\mathbf{4}}$, which hints at the increasing complexity of formulas for larger $k$. % To with, $\ell_\mathbf{4}(x)$ is $3$ if $x=(ij)$ and is given by 
Thus it is arguably even harder to describe cover relations for the orders $\leq_\mathbf{k}$, and the rest of the general structure of the poset is probably even trickier.

There is, however, a certain subclass of permutations for which the results generalize. By writing a $k$-cycle as a product of $k-1$ transpositions, one has easily $(k-1)\ell_\mathbf{k}(x)\geq \lt(x)$.  Equality holds in the case described by the following easy proposition.

\begin{proposition}\label{prop:nice_k_length}
	A permutation $x$ satisfies $\ell_\mathbf{k}(x)= \frac{\lt(x)}{k-1}$ if and only if all its cycles have length congruent to $1$ modulo $k-1$.
\end{proposition}
 
Notice that for $k=3$ the elements occurring in Proposition~\ref{prop:nice_k_length} are precisely the elements of $\Alto$, and the results from Section~\ref{sec:alternating} concerning $\Alto$ can be extended to the analogous set for $k\geq 4$. 

Furthermore, the enumerative results of Sections \ref{sec:enumerative} and \ref{sec:bijective} can also be generalized.  The set of all elements $x$ in the subgroup of $\Sym_{(k-1)n+1}$ generated by all $k$-cycles, which satisfy $x\leq_{\mathbf{k}}\bigl(1\;2\;\ldots\;(k-1)n+1\bigr)$ is then in bijection with the set of pairs of $k$-ary trees which have a total of $n$ internal nodes.  Therefore, the cardinality of these sets is given by $\frac{2}{(k-1)n+2}\binom{kn+1}{n}$.  The zeta polynomial of the resulting poset is given by $\frac{q}{(q-1)(k-1)n+q}\binom{(q-1)(k-1)n+q+n-1}{n}$. These results and others can be found in the preprint \cite{muehle19kindivisible} of the authors together with Nathan Williams.

\subsection{Generalization to Coxeter Groups}
	\label{sec:coxeter}
Let $(W,S)$ be a finite irreducible Coxeter system of rank $n$, and let $T=\{w^{-1}sw\mid w\in W,s\in S\}$ denote the set of all reflections of $W$. The alternating group $\Alt(W)$ is the subset of $W$ of elements $w$ such that $\ell_S(w)$  (or, equivalently, $\ell_T(w)$)  is even. Now consider
\begin{equation}\label{eq:alt_coxeter_generators}
	A = \bigl\{w^{-1}stw\mid w\in W, s,t \in S, 3\leq m_{st}\leq \infty \bigr\}.
\end{equation}
This corresponds to the set of $3$-cycles for $W=\Sym_n$. 

In general {\em the set $A$ generates $\Alt(W)$}. In fact the subset $A'\subseteq A$ of elements $st$ with $m_{st}\geq 3$ already generates $\Alt(W)$:  since by definition $\Alt(W)$ is generated by all products $st$, we just need to show that when $m_{st}=2$ then $st$ is a product of elements of $A'$. But since $W$ is irreducible, for any such $s,t$ there is a path $s=s_0\to s_1\to \cdots \to s_k=t$ in the Coxeter graph, which means that $m_{s_{i-1}s_i}\geq 3$ for all $i$. Now $st=(s_0s_1)(s_1s_2)\cdots (s_{k-1}s_k)$ is a product of elements of $A'$ as wanted.

We can thus define $\ell_{A}$ to be the word length in $\Alt(W)$ with respect to $A$ and the $A$-prefix order $\leq_{A}$. The structural and enumerative questions that we dealt with in type $A$ can thus be studied for any finite Coxeter group.

In particular, recall that a \defn{Coxeter element} $c$ of $W$ is a product of any permutation of the Coxeter generators, and therefore has $\ell_{T}(c)=n$.  Hence $c\in\Alt(W)$ if and only if $W$ has even rank.  For any Coxeter element $c\in W$, we define the set
\begin{equation}\label{eq:enc_coxeter}
	\enc_{W}(c) = \{x\in\Alt(W)\mid x\leq_{A}c\},
\end{equation}
and we denote by $\penc_{W}(c)=\bigl(\enc_{W}(c),\leq_{A}\bigr)$ the corresponding poset. Since any two Coxeter elements $c,c'\in W$ are conjugate, the posets $\penc_{W}(c)$ and $\penc_{W}(c')$ are isomorphic. We are still missing a good combinatorial model for these posets, but early computations show that for instance in type $B$, the zeta polynomial of $\penc_{B_{2n}}$ factors nicely.

%\bibliography{literature}

\end{document}